\documentclass[hidelinks,onefignum,onetabnum]{siamart251216}
\usepackage[T1]{fontenc}
\usepackage{lmodern,amssymb,mathrsfs,microtype,graphicx}
\usepackage{xurl,placeins}
\newsiamremark{remark}{Remark}
\newcommand{\T}{\mathbb T}
\newcommand{\E}{\mathbb E}
\newcommand{\PP}{\mathbb P}
\newcommand{\F}{\mathcal F}
\newcommand{\HS}{\mathcal L_2}
\newcommand{\U}{\mathfrak U}
\newcommand{\ip}[2]{\langle #1,#2\rangle}
\newcommand{\norm}[1]{\lVert #1\rVert}
\newcommand{\dd}{\,\mathrm d}
\newcommand{\ind}{\mathbf 1}
\newcommand{\compactembed}{\mathrel{\lhook\joinrel\hookrightarrow}}
\allowdisplaybreaks[2]
\headers{Convergence rate for stochastic Korteweg--de Vries}{J. Li, X. Niu, and C. Li}
\title{Strong convergence rate of the implicit midpoint method for stochastic KdV equations\thanks{This work was partially supported by the National Natural Science Foundation of China (Grant Nos.~11871271, 12001283, and 11901347).}}
\author{Junjie Li\thanks{School of Mathematics, Nanjing University, Nanjing 210093, P. R. China\newline(\email{652023210010@smail.nju.edu.cn}).}
\and Xinyan Niu\thanks{School of Statistics and Mathematics, Shandong University of Finance and Economics, Jinan, Shandong 250014, P. R. China (\email{20153462@sdufe.edu.cn}).}
\and Chun Li\thanks{Corresponding author. School of Mathematics, Nanjing University, Nanjing 210093, P. R. China (\email{lichun@nju.edu.cn}).}}
\hypersetup{pdftitle={Strong convergence rate of the implicit midpoint method for stochastic KdV equations},pdfauthor={Junjie Li, Xinyan Niu, and Chun Li}}
\begin{document}
\maketitle
\begin{abstract}
We establish a strong convergence rate for the implicit midpoint method applied to the periodic stochastic Korteweg--de Vries equation with additive noise, addressing a question that has remained open since the work of Debussche and Printems. To control nonlinear error growth, we obtain exponential integrability estimates for the exact and comparison solutions using the coercivity of a classical KdV invariant. We then derive local error bounds from estimates for the comparison solutions and control error accumulation through an $L^2$ error propagation identity that isolates a martingale increment. Under $H_0^{9/2}$ regularity assumptions on the noise and the deterministic initial data, we prove that, on every fixed finite time interval, the maximum $L^2(\T)$ error at the time points is $O(h)$ in $L^p(\Omega)$ for every $1\le p<\infty$. Notably, the exponential integrability and the convergence rate hold for every damping coefficient $\lambda\ge0$, including the undamped case $\lambda=0$. Numerical experiments confirm the predicted first-order convergence.
\end{abstract}


\section{Introduction}\label{sec:introduction}
We consider the periodic stochastic Korteweg--de Vries (sKdV) equation
\begin{equation}\label{eq:spde}
 \dd u+\left(\partial_x^3u+\tfrac12\partial_x(u^2)+\lambda u\right)\dd t
 =\Phi\dd W,\qquad u(0)=u_0,\qquad \lambda\ge0,
\end{equation}
on $\T=\mathbb R/(2\pi\mathbb Z)$. We work on a filtered probability space $(\Omega,\F,(\F_t)_{t\ge0},\PP)$ satisfying the usual conditions. Here $W$ is a cylindrical Wiener process on a separable real Hilbert space $\U$ with respect to $(\F_t)_{t\ge0}$, and $\Phi$ is a deterministic Hilbert--Schmidt operator with values in a periodic Sobolev space. The initial data and the range of $\Phi$ have zero spatial mean. Let $\HS(\U,X)$ denote the space of Hilbert--Schmidt operators from $\U$ to a Hilbert space $X$. The well-posedness and long-time behavior of \eqref{eq:spde} have been studied extensively. De~Bouard, Debussche, and Tsutsumi~\cite{deBouardDebusscheTsutsumi2005} proved local existence and uniqueness for the periodic equation with additive noise that is white in time and almost white in space; Oh~\cite{Oh2009} established local well-posedness for additive space-time white noise. For the weakly damped periodic equation, Glatt-Holtz, Martinez, and Richards~\cite{GHMR2026} studied the regularity and ergodic properties of invariant measures and proved exponential mixing under sufficiently large damping. For arbitrary positive damping, they used modifications of the KdV invariants to obtain uniform-in-time Sobolev moment bounds and exponential moment estimates.

Debussche and Printems~\cite{DebusschePrintems1999} used the implicit midpoint method to investigate the influence of homogeneous noise on soliton evolution. They later proved convergence for the equation on the real line with additive, spatially localized noise~\cite{DebusschePrintems2006}, without obtaining a convergence rate. Further numerical studies have examined different aspects of sKdV approximation. Y.~Li, Shu, and Tang~\cite{LiShuTang2020} developed an ultra-weak discontinuous Galerkin method with implicit--explicit time stepping for generalized stochastic KdV equations. D'Ambrosio and Di~Giovacchino~\cite{DAmbrosioDiGiovacchino2023} studied numerical conservation properties of stochastic $\theta$-methods, while Cui et~al.~\cite{CuiEtAl2026} derived strong error bounds for a temporal approximation based on small-noise linearization. An explicit exponential integrator with artificial viscosity was studied by J.~Li and C.~Li~\cite{LiLiExponential}, and Galilean splitting methods were analyzed at the Fourier--Galerkin level by J.~Li, Z.~Wang, and C.~Li~\cite{LiWangLiGalilean}. J.~Li et~al.~\cite{LiLiZhouWang2026} analyzed a Lie--Trotter splitting with exact subflows, proving strong and weak order~$1$ uniformly in time under sufficiently large damping and strong order~$1$ on finite time intervals under a weaker damping condition.

Strong convergence rates have been established for numerical approximations of a wide range of SPDEs. Abdulle, Br\'ehier, and Vilmart~\cite{AbdulleBrehierVilmart2023} and X.~Wang~\cite{Wang2020} obtained such rates for parabolic semilinear SPDEs and the stochastic Allen--Cahn equation, respectively, using the smoothing properties of the associated analytic semigroups. Y.~Wang, W.~Cao, and Y.~Cao~\cite{WangCaoCao2026} established strong convergence rates for a fully discrete tamed exponential scheme for the stochastic Burgers--Huxley equation. Chen, Hong, and Ji~\cite{ChenHongJi2019} employed discrete a priori estimates for stochastic Maxwell equations, while Cui et~al.~\cite{CuiHongLiuZhou2019} used exponential integrability for stochastic cubic Schr\"odinger equations. Br\'ehier and Cohen~\cite{BrehierCohen2023} used exponential moment bounds to establish strong convergence rates for a splitting scheme for nonlinear stochastic Schr\"odinger equations with additive noise and smooth nonlocal interactions. Jentzen and Pu\v{s}nik~\cite{JentzenPusnik2018} established exponential moment bounds for tamed exponential Euler approximations of nonlinear SPDEs. For \eqref{eq:spde}, however, the Airy group preserves Sobolev norms and provides no dissipative smoothing, while the nonlinear transport term carries one spatial derivative. The implicit discretization of the nonlinear term further complicates the convergence analysis. To the best of our knowledge, the question of a strong convergence rate for the implicit midpoint method has remained open.

For the periodic problem \eqref{eq:spde}, we obtain such a rate for this method, which preserves the quadratic invariant $\norm{u}_{L^2}^2$ in the absence of noise and damping. The scheme reads:
\begin{equation}\label{eq:scheme}
 \begin{split}
 U^{n+1}-U^n+h\left(\partial_x^3m^n+
             \tfrac12\partial_x((m^n)^2)+\lambda m^n\right)
 &=\Phi\Delta W^n,\\
 m^n=\tfrac12(U^{n+1}+U^n),\qquad U^0&=u_0,
 \end{split}
\end{equation}
where $h>0$ is the time step, $t_n=nh$ and $\Phi\Delta W^n=\int_{t_n}^{t_{n+1}}\Phi\dd W$.

To control nonlinear error growth, we establish exponential integrability for the comparison solutions using a coercivity estimate for a classical KdV invariant~\cite{Lax1968}, and pass to the limit to obtain it for the exact solution. For deterministic initial data $u_0\in H_0^{9/2}(\T)$ and a noise coefficient $\Phi\in\HS(\U,H_0^{9/2}(\T))$, Theorem~\ref{thm:exact} gives, for every $T>0$ and $0\le\Lambda<\infty$,
\[
 \sup_{0\le\lambda\le\Lambda}
 \E\exp\!\left(A\int_0^T\norm{\partial_xu(t)}_\infty\dd t\right)<\infty
 \qquad\text{for every }A>0,
\]
together with Sobolev moment bounds and a comparison estimate. In particular, we extend the corresponding exponential integrability result of Glatt-Holtz, Martinez, and Richards~\cite{GHMR2026} to the undamped equation. Combining local error estimates for the comparison solutions, an $L^2$ error propagation identity that isolates a martingale increment, and the comparison estimate yields Theorem~\ref{thm:main}: for every $r\ge1$,
\[
 \E\max_{0\le n\le\lfloor T/h\rfloor}
       \norm{u(t_n)-U^n}_{L^2(\T)}^{2r}\le Ch^{2r}
\]
for all $0<h\le1$, $\lambda\in[0,\Lambda]$, and every solution of the implicit midpoint scheme~\eqref{eq:scheme}, with $C$ independent of $h$ and $\lambda$. The estimate also covers the undamped case, whereas the corresponding result for the splitting method in~\cite{LiLiZhouWang2026} requires a positive lower bound on the damping coefficient. Numerical experiments demonstrate first-order convergence.

The remainder of the paper is organized as follows. Section~\ref{sec:midpoint} establishes existence and moment bounds for solutions of the implicit midpoint scheme. Section~\ref{sec:exact} proves exponential integrability and Sobolev moment bounds for the comparison solutions and the exact solution, together with the comparison estimate. Section~\ref{sec:strong} derives the local error estimates and the strong error estimate. Section~\ref{sec:numerics} presents numerical experiments, and Section~\ref{sec:conclusion} concludes the paper.

\section{The implicit midpoint method}\label{sec:midpoint}
We use the normalized Lebesgue measure $\mathrm dx/(2\pi)$ on $\T$ and write $\ip fg=\int_\T fg,\ \norm v_p=\norm v_{L^p(\T)}$. For real-valued periodic functions, with $\widehat v(k)=\int_\T v(x)e^{-ikx}$, set
\[
 \norm v_{H^q}^2=\sum_{k\in\mathbb Z}(1+k^2)^q|\widehat v(k)|^2,
 \qquad H^q_0=\{v\in H^q:\widehat v(0)=0\},\qquad H=L^2_0.
\]

A solution of \eqref{eq:scheme} consists of $\F_{t_n}$-measurable $H$-valued random variables $U^n$ and $\F_{t_{n+1}}$-measurable $H^3_0$-valued random variables $m^n$, $n\ge0$, such that $U^0=u_0$, $m^n=(U^{n+1}+U^n)/2$, and \eqref{eq:scheme} holds in $H$ almost surely at every step.

With $m=m^n$ and $f=U^n+\Phi\Delta W^n/2$, equation~\eqref{eq:scheme} is equivalent to
\begin{equation}\label{eq:stage}
 m+\frac h2\left(\partial_x^3m+m\partial_x m+\lambda m\right)=f.
\end{equation}
Define $Q=(I+\tfrac h2(\partial_x^3+\lambda I))^{-1}$. The Fourier multiplier of $Q$ is $(1+\frac h2(\lambda-ik^3))^{-1}$, and $Q^*$ has the complex conjugate multiplier. Hence
\begin{equation}\label{eq:Q-contraction}
 \norm Q_{H^q\to H^q}\le1,\qquad
 \norm{Q^*}_{H^q\to H^q}\le1,\qquad q\in\mathbb R.
\end{equation}

\begin{proposition}[Existence of solutions]\label{prop:roots}
For each $h>0$ and $\lambda\ge0$, the solution set of \eqref{eq:stage} in $H^3_0$ is nonempty and compact for every $f\in H$. There is a Borel measurable map $\mathcal S_{h,\lambda}:H\to H^3_0$ such that $\mathcal S_{h,\lambda}(f)$ solves \eqref{eq:stage} for every $f\in H$. Every solution $m$ satisfies
\begin{equation}\label{eq:radial}
 \norm m_2\le(1+\lambda h/2)^{-1}\norm f_2\le\norm f_2.
\end{equation}
If $u_0\in H$ is deterministic and $\Phi\in\HS(\U,H)$, then \eqref{eq:scheme} has a solution. If $u_0\in H^3_0$, then every solution satisfies $U^n\in H^3_0$ almost surely for every $n\ge0$.
\end{proposition}
\begin{proof}
\medskip\noindent\emph{Step 1: compactness of $\mathcal T_f$.}
Fix $h>0$, $\lambda\ge0$ and $f\in H$, and define
\[
 \mathcal T_f(m)=Qf-\frac h4Q\partial_x(m^2),\qquad m\in H.
\]
For $v\in L^2$ and $g\in L^1$, Parseval's identity and $|\widehat g(k)|\le\norm g_1$ give
\[
 \begin{aligned}
 \norm{Qv}_{H^3}^2
 &=\sum_{k\in\mathbb Z}
 \frac{(1+k^2)^3|\widehat v(k)|^2}{(1+\lambda h/2)^2+\frac{h^2}{4}k^6}
 \le C_{h,\lambda}\norm v_2^2,\\
 \norm{Q\partial_x g}_{H^1}^2
 &\le\norm g_1^2\sum_{k\ne0}
 \frac{(1+k^2)k^2}{(1+\lambda h/2)^2+\frac{h^2}{4}k^6}
 \le C_{h,\lambda}\norm g_1^2.
 \end{aligned}
\]
The second series converges by comparison with $\sum_{k\ne0}|k|^{-2}$. Since $f$ has mean zero and
\[
 \norm{m^2-\widetilde m^2}_1
 \le(\norm m_2+\norm{\widetilde m}_2)\norm{m-\widetilde m}_2,
\]
the map $\mathcal T_f:H\to H^1_0$ is continuous and maps bounded sets into bounded sets. The compact embedding $H^1_0\compactembed H$ therefore makes $\mathcal T_f:H\to H$ compact.

\medskip\noindent\emph{Step 2: a priori estimate and existence.}
Let $m\in H$ satisfy $m=t\mathcal T_f(m)$ for some $t\in[0,1]$. Step~1 gives $m\in H^1_0$. Since $\norm{m\partial_x m}_2\le C\norm m_{H^1}^2$, we have
\[
 m=t\mathcal T_f(m)
 =Q\left(tf-\frac{th}{2}m\partial_x m\right)\in H^3_0.
\]
Applying $Q^{-1}$, taking the $L^2$ inner product with $m$, and integrating by parts yield
\[
 (1+\lambda h/2)\norm m_2^2=t\ip fm\le t\norm f_2\norm m_2.
\]
Thus $\norm m_2\le(1+\lambda h/2)^{-1}\norm f_2$, uniformly in $t\in[0,1]$. The Leray--Schauder fixed-point theorem~\cite{Deimling1985} gives a fixed point in $H^3_0$ solving \eqref{eq:stage}. The same calculation with $t=1$ proves \eqref{eq:radial} for every solution.

\medskip\noindent\emph{Step 3: compactness of the solution set and measurable selection.}
For any solution $m$, Step~1 and \eqref{eq:radial} give
\[
 \norm m_{H^1}
 \le C_{h,\lambda}(\norm f_2+\norm m_2^2)
 \le C_{h,\lambda}(\norm f_2+\norm f_2^2).
\]
Using $Q:L^2\to H^3$ and $\norm{m\partial_x m}_2\le C\norm m_{H^1}^2$, we obtain
\[
 \begin{aligned}
 \norm m_{H^3}
 &=\norm{Q\left(f-\frac h2m\partial_x m\right)}_{H^3}
 \le C_{h,\lambda}(\norm f_2+\norm m_{H^1}^2)\\
 &\le C_{h,\lambda}(\norm f_2+\norm f_2^2+\norm f_2^4).
 \end{aligned}
\]
Let $(m_j)$ be a sequence of solutions for the fixed $f$. Since $H^3_0\compactembed H^2_0$, a subsequence satisfies $m_j\to m$ in $H^2_0$, and
\[
 \norm{m_j\partial_x m_j-m\partial_x m}_2
 \le\norm{m_j-m}_\infty\norm{\partial_x m_j}_2
       +\norm m_\infty\norm{\partial_x m_j-\partial_x m}_2\longrightarrow0.
\]
The bound on $Q:L^2\to H^3$ then yields
\[
 m_j=Qf-\frac h2Q(m_j\partial_x m_j)\longrightarrow Qf-\frac h2Q(m\partial_x m)
 \quad\hbox{in }H^3_0.
\]
Uniqueness of the $H^2$ limit gives $m=Qf-\frac h2Q(m\partial_x m)$. Thus $m$ solves \eqref{eq:stage}, and the solution set is compact in $H^3_0$. The map
\[
 H\times H^3_0\ni(f,m)\longmapsto
 m+\frac h2(\partial_x^3m+m\partial_x m+\lambda m)-f\in H
\]
is continuous, so the solution graph is closed and hence Borel. Its sections are nonempty and compact, and both spaces are Polish. The Arsenin--Kunugui theorem~\cite[Theorem 18.18]{Kechris1995} therefore yields the required Borel selector $\mathcal S_{h,\lambda}$.

Set $U^0=u_0$ and, for $n\ge0$, define
\[
 m^n=\mathcal S_{h,\lambda}(U^n+\Phi\Delta W^n/2),\qquad
 U^{n+1}=2m^n-U^n.
\]
The Borel measurability of $\mathcal S_{h,\lambda}$ and the continuous embedding $H^3_0\hookrightarrow H$ give the required measurability by induction. The recursion satisfies \eqref{eq:scheme}, so these sequences form a solution. For every solution with $u_0\in H^3_0$, the identity $U^{n+1}=2m^n-U^n$ also gives $U^n\in H^3_0$ almost surely for all $n\ge0$ by induction.
\end{proof}

\begin{proposition}[Finite-time moment bounds]\label{prop:numerical}
For $p\ge1$, $T>0$, $\Phi\in\HS(\U,H^2_0)$ and deterministic $u_0\in H$, there is a constant $C_{p,T,\Phi}$, depending on $\Phi$ only through $\norm\Phi_{\HS(\U,H^2)}$, such that, for every $\lambda\ge0$, every $0<h\le1$ and every solution of \eqref{eq:scheme},
\begin{equation}\label{eq:numerical}
 \E\max_{0\le n\le\lfloor T/h\rfloor}\norm{U^n}_2^{2p}
 \le C_{p,T,\Phi}(1+\norm{u_0}_2^{2p}).
\end{equation}
\end{proposition}
\begin{proof}
Taking the $L^2$ inner product of \eqref{eq:scheme} with $U^{n+1}+U^n=2m^n\in H^3_0$ and integrating by parts, we obtain
\begin{equation}\label{eq:numerical-energy}
 \norm{U^{n+1}}_2^2-\norm{U^n}_2^2+2\lambda h\norm{m^n}_2^2
 =\ip{\Phi\Delta W^n}{U^{n+1}+U^n}.
\end{equation}
Since $\lambda\ge0$, applying the Cauchy--Schwarz inequality to \eqref{eq:numerical-energy} yields $\norm{U^{n+1}}_2\le\norm{U^n}_2+\norm{\Phi\Delta W^n}_2$. Let $N=\lfloor T/h\rfloor$. Iteration and Gaussian moments imply $\E\max_{0\le n\le N}\norm{U^n}_2^{2p}<\infty$ for every fixed $h>0$.

We separate the stochastic term in \eqref{eq:numerical-energy} into a martingale increment and terms controlled by the noise moments. Applying $Q$ to \eqref{eq:stage} with $f=U^n+\Phi\Delta W^n/2$ gives
\[
 m^n=QU^n+\frac12Q\Phi\Delta W^n-\frac h2Q(m^n\partial_x m^n).
\]
Substituting this identity into \eqref{eq:numerical-energy} and integrating the nonlinear term by parts, we obtain
\begin{equation}\label{eq:resolved-moment}
 \begin{split}
 \norm{U^{n+1}}_2^2-\norm{U^n}_2^2
 ={}&-2\lambda h\norm{m^n}_2^2
       +2\ip{\Phi\Delta W^n}{QU^n}\\
 &+\ip{\Phi\Delta W^n}{Q\Phi\Delta W^n}
       +\frac h2\int_\T(\partial_x Q^*\Phi\Delta W^n)(m^n)^2.
 \end{split}
\end{equation}
By \eqref{eq:radial} and \eqref{eq:Q-contraction},
\[
 \norm{m^n}_2^2\le2\norm{U^n}_2^2+\frac12\norm{\Phi\Delta W^n}_2^2,
 \qquad
 \bigl|\ip{\Phi\Delta W^n}{Q\Phi\Delta W^n}\bigr|
 \le\norm{\Phi\Delta W^n}_2^2.
\]
Hence, discarding the nonpositive damping term in \eqref{eq:resolved-moment}, we obtain
\begin{equation}\label{eq:moment-recursion}
 \begin{split}
 \norm{U^{n+1}}_2^2-\norm{U^n}_2^2
 \le{}&2\ip{\Phi\Delta W^n}{QU^n}
 +h\norm{\partial_x Q^*\Phi\Delta W^n}_\infty\norm{U^n}_2^2\\
 &+\left(1+\frac h4\norm{\partial_x Q^*\Phi\Delta W^n}_\infty\right)
       \norm{\Phi\Delta W^n}_2^2.
 \end{split}
\end{equation}

Set $E_j^*=\max_{0\le n\le j}\norm{U^n}_2^2$ for $0\le j\le N$. Summing \eqref{eq:moment-recursion} and taking the maximum over $0\le i\le j$ gives
\begin{multline}\label{eq:numerical-maximum}
 E_j^*\le\norm{u_0}_2^2+
       \max_{0\le i\le j}\Biggl|2\sum_{0\le n<i}\ip{\Phi\Delta W^n}{QU^n}\Biggr|\\
 +h\sum_{0\le n<j}\norm{\partial_x Q^*\Phi\Delta W^n}_\infty\norm{U^n}_2^2
 +\sum_{0\le n<j}
       \left(1+\frac h4\norm{\partial_x Q^*\Phi\Delta W^n}_\infty\right)
       \norm{\Phi\Delta W^n}_2^2.
\end{multline}
Since $U^n$ is $\F_{t_n}$-measurable and $\Phi\Delta W^n$ is independent of $\F_{t_n}$ and centered, the partial sums $2\sum_{0\le n<i}\ip{\Phi\Delta W^n}{QU^n}$ form a martingale with respect to $(\F_{t_i})$. The discrete Burkholder--Davis--Gundy (BDG) inequality, H\"older's inequality and Gaussian moments yield
\begin{equation}\label{eq:numerical-martingale}
 \begin{aligned}
 &\Biggl\|\max_{0\le i\le j}
       \Biggl|2\sum_{0\le n<i}\ip{\Phi\Delta W^n}{QU^n}\Biggr|\Biggr\|_{L^p(\Omega)}\\
 &\quad\le C_p\Biggl\|\Biggl(\sum_{0\le n<j}\norm{\Phi\Delta W^n}_2^2
                              \norm{U^n}_2^2\Biggr)^{1/2}\Biggr\|_{L^p(\Omega)}
 \le C_{p,\Phi}(jh)^{1/2}\norm{E_j^*}_{L^p(\Omega)}^{1/2}.
 \end{aligned}
\end{equation}
Here we used ${\bigl\|\sum_{0\le n<j}\norm{\Phi\Delta W^n}_2^2\bigr\|_{L^p(\Omega)}\le C_{p,\Phi}jh}$, together with $\norm{QU^n}_2\le\norm{U^n}_2$ and $\norm{U^n}_2^2\le E_j^*$ for $n\le j$.

The bound ${\norm{\partial_x Q^*\Phi\Delta W^n}_{L^p(\Omega;L^\infty)}\le C_{p,\Phi}h^{1/2}}$ follows from $H^2\hookrightarrow W^{1,\infty}$, \eqref{eq:Q-contraction} and Gaussian moments, uniformly for $\lambda\ge0$. Using the independence of $U^n$ and $\Phi\Delta W^n$ and H\"older's inequality, we obtain, for $h\le1$,
\begin{equation}\label{eq:numerical-noise-terms}
 \begin{aligned}
 &\left\|h\norm{\partial_x Q^*\Phi\Delta W^n}_\infty\norm{U^n}_2^2\right\|_{L^p(\Omega)}
 \le C_{p,\Phi}h^{3/2}\norm{E_n^*}_{L^p(\Omega)},\\
 &\left\|\left(1+\frac h4\norm{\partial_x Q^*\Phi\Delta W^n}_\infty\right)
                 \norm{\Phi\Delta W^n}_2^2\right\|_{L^p(\Omega)}
 \le C_{p,\Phi}(h+h^{5/2})\le C_{p,\Phi}h.
 \end{aligned}
\end{equation}
Taking $L^p(\Omega)$ norms in \eqref{eq:numerical-maximum} and using \eqref{eq:numerical-martingale}--\eqref{eq:numerical-noise-terms} gives
\[
 \begin{aligned}
 \norm{E_j^*}_{L^p(\Omega)}
 \le{}&\norm{u_0}_2^2+C(jh)^{1/2}\norm{E_j^*}_{L^p(\Omega)}^{1/2}\\
 &+Ch^{3/2}\sum_{0\le n<j}\norm{E_n^*}_{L^p(\Omega)}+Cjh.
 \end{aligned}
\]
Since $\norm{E_j^*}_{L^p(\Omega)}<\infty$, we apply Young's inequality and then discrete Gronwall, using $jh\le T$ and $Nh^{3/2}\le T$, to obtain $\norm{E_N^*}_{L^p(\Omega)}\le C_{p,T,\Phi}(1+\norm{u_0}_2^2)$. Taking the $p$th power yields \eqref{eq:numerical}, with a constant independent of $h$, $\lambda\ge0$ and the choice of solution.
\end{proof}
\section{Exponential integrability and comparison estimates}\label{sec:exact}

\subsection{Coercivity of a classical KdV invariant}

For $v\in H^2$, consider the classical KdV invariant~\cite{Lax1968} (see also~\cite[(2.6)]{GHMR2026}):
\begin{equation}\label{eq:J}
 \mathcal J(v)=\int_\T\left((\partial_x^2 v)^2-\frac53v(\partial_x v)^2+\frac5{36}v^4\right).
\end{equation}

\begin{lemma}[Positive decomposition and coercivity]\label{lem:sos}
For every real-valued $v\in H^2(\T)$, set $\rho=(|\partial_x v|^2/60)^{1/3}$. Then
\begin{equation}\label{eq:sos}
 \mathcal J(v)=\int_\T\bigl(\partial_x^2 v+\tfrac13v^2+v\rho-8\rho^2\bigr)^2
 +\int_\T\bigl(\tfrac16v^2-2v\rho+\rho^2\bigr)^2+7\int_\T\rho^4.
\end{equation}
There exist constants $c,C>0$ such that
\begin{equation}\label{eq:coercive}
 c\left(\norm{\partial_x^2 v}_2^2+\norm v_4^4\right)
 \le\mathcal J(v)
 \le C\left(\norm{\partial_x^2 v}_2^2+\norm v_4^4\right).
\end{equation}
\end{lemma}
\begin{proof}
For smooth $v$, the identity $(\partial_x v)^2=60\rho^3$ and periodic integration by parts give
\[
 \int_\T v^2\partial_x^2 v=-120\int_\T v\rho^3,
 \qquad
 \int_\T v\rho \partial_x^2 v=-36\int_\T\rho^4,
 \qquad
 \int_\T\rho^2\partial_x^2 v=0.
\]
The last two identities follow from the chain rule and periodic integration by parts. Expanding the squares proves \eqref{eq:sos}, which extends to $H^2$ by a density argument using the embedding $H^2\hookrightarrow C^1$.

Positivity and homogeneity of the quartic form, together with compactness of the unit $\ell^4$-sphere, give a constant $c_0>0$ such that
\[
 \left(\frac16v^2-2v\rho+\rho^2\right)^2+7\rho^4
 \ge c_0(v^4+\rho^4).
\]
Using $\norm v_4^4+\norm\rho_4^4\le C\mathcal J(v)$ together with the first integral in \eqref{eq:sos}, we obtain
\[
 \norm{\partial_x^2 v}_2^2+\norm v_4^4
 \le2\int_\T\left(\partial_x^2 v+\frac13v^2+v\rho-8\rho^2\right)^2
       +C\left(\norm v_4^4+\norm\rho_4^4\right)
 \le C\mathcal J(v).
\]
For the upper bound, integration by parts and the Cauchy--Schwarz inequality give
\[
 \mathcal J(v)=\norm{\partial_x^2 v}_2^2+\frac56\int_\T v^2\partial_x^2 v
                   +\frac5{36}\norm v_4^4
 \le C\left(\norm{\partial_x^2 v}_2^2+\norm v_4^4\right).
\]
\end{proof}

Define
\begin{equation}\label{eq:VG}
 V(v)=1+\mathcal J(v),\qquad G(v)=V(v)^{3/7}.
\end{equation}
By Lemma~\ref{lem:sos}, $V\ge1$, so $G$ is a smooth functional on $H^2$.

\begin{lemma}[Control of the spatial derivative]\label{lem:derivative-control}
There exists $C>0$ such that
\begin{equation}\label{eq:derivative-bound}
 \norm{\partial_x v}_\infty\le C G(v),\qquad v\in H^2.
\end{equation}
\end{lemma}
\begin{proof}
The Gagliardo--Nirenberg inequality, \eqref{eq:coercive} and $V\ge1$ give
\[
 \begin{aligned}
 \norm{\partial_x v}_\infty
 &\le C\norm{\partial_x^2 v}_2^{5/7}\norm v_4^{2/7}+C\norm v_4\\
 &\le C\bigl(V^{1/2}\bigr)^{5/7}\bigl(V^{1/4}\bigr)^{2/7}+CV^{1/4}
 \le CV^{3/7}=CG(v).
 \end{aligned}
\]
\end{proof}

Assume $\Phi\in\HS(\U,H^2_0)$ and let $(e_j)$ be an orthonormal basis of $\U$. We denote first and second Fr\'echet derivatives by $\mathrm D$ and $\mathrm D^2$.

\begin{lemma}[Bounds for the drift and quadratic-variation density]\label{lem:generator}
There exist constants $C_0,C_1,C_2>0$ such that for every $\lambda\ge0$ and smooth real-valued $v$,
\begin{equation}\label{eq:generator}
 \begin{aligned}
 &-\mathrm DG(v)[\partial_x^3v+v\partial_x v+\lambda v]
    +\frac12\sum_j\mathrm D^2G(v)[\Phi e_j,\Phi e_j]\\
 &\qquad\le C_0\norm\Phi_{\HS(\U,H^2)}^2+C_1\lambda G(v),\\
 &\sum_j|\mathrm DG(v)[\Phi e_j]|^2
    \le C_2\norm\Phi_{\HS(\U,H^2)}^2G(v)^{5/6}.
 \end{aligned}
\end{equation}
\end{lemma}
\begin{proof}
For any real-valued $\phi\in H^2$, differentiation and integration by parts give
\begin{align}
 \mathrm D\mathcal J(v)[\phi]
 ={}&2\ip{\partial_x^2 v}{\partial_x^2 \phi}-\frac53\int_\T\phi(\partial_x v)^2
          +\frac53\int_\T v^2\partial_x^2 \phi+\frac59\int_\T v^3\phi,\label{eq:DJ}\\
 \mathrm D^2\mathcal J(v)[\phi,\phi]
 ={}&2\norm{\partial_x^2 \phi}_2^2+\frac{10}{3}\int_\T v(\partial_x \phi)^2
          +\frac{20}{3}\int_\T v\phi\partial_x^2 \phi+\frac53\int_\T v^2\phi^2.\label{eq:DDJ}
\end{align}
Set $\phi_j=\Phi e_j$. By \eqref{eq:DDJ}, $H^2\hookrightarrow W^{1,\infty}$ and \eqref{eq:coercive}, the trace series converges absolutely and satisfies
\begin{equation}\label{eq:hessian}
 \sum_j|\mathrm D^2\mathcal J(v)[\phi_j,\phi_j]|
 \le C\norm\Phi_{\HS(\U,H^2)}^2(1+\norm v_2+\norm v_2^2)
 \le C\norm\Phi_{\HS(\U,H^2)}^2V^{1/2}.
\end{equation}
Setting $\phi=v$ in \eqref{eq:DJ} and integrating by parts gives
\[
 \mathrm D\mathcal J(v)[v]=2\norm{\partial_x^2 v}_2^2-5\int_\T v(\partial_x v)^2+\frac59\norm v_4^4.
\]
Using $\int_\T v(\partial_x v)^2=-\frac12\int_\T v^2\partial_x^2 v$ and \eqref{eq:coercive} yields $|\mathrm D\mathcal J(v)[v]|\le CV$.

Since $G=V^{3/7}$, the chain rule gives
\[
 \begin{aligned}
 \mathrm DG(v)[\phi]
 &=\frac37V^{-4/7}\mathrm D\mathcal J(v)[\phi],\\
 \mathrm D^2G(v)[\phi,\phi]
 &=\frac37V^{-4/7}\mathrm D^2\mathcal J(v)[\phi,\phi]
   -\frac{12}{49}V^{-11/7}|\mathrm D\mathcal J(v)[\phi]|^2.
 \end{aligned}
\]
Conservation of $\mathcal J$ under the deterministic KdV flow gives $\mathrm D\mathcal J(v)[\partial_x^3v+v\partial_x v]=0$, so the KdV contribution vanishes. Dropping the nonpositive term in $\mathrm D^2G$ and using \eqref{eq:hessian}, $|\mathrm D\mathcal J(v)[v]|\le CV$, and $V\ge1$, we obtain
\[
 \begin{aligned}
 &-\mathrm DG(v)[\partial_x^3v+v\partial_x v+\lambda v]
    +\frac12\sum_j\mathrm D^2G(v)[\phi_j,\phi_j]\\
 &\quad\le C\lambda V^{3/7}+C\norm\Phi_{\HS(\U,H^2)}^2V^{-1/14}
 \le C_1\lambda G(v)+C_0\norm\Phi_{\HS(\U,H^2)}^2.
 \end{aligned}
\]

Using $\norm{\partial_x v}_2^2\le\norm v_2\norm{\partial_x^2 v}_2$, $\norm v_2\le\norm v_4$, \eqref{eq:DJ}, the triangle inequality in $\ell^2$, and \eqref{eq:coercive}, we obtain
\begin{align}\label{eq:gradient}
 \left(\sum_j|\mathrm D\mathcal J(v)[\phi_j]|^2\right)^{1/2}
 &\le C\norm\Phi_{\HS(\U,H^2)}
       (\norm{\partial_x^2 v}_2+\norm{\partial_x v}_2^2+\norm v_4^2+\norm v_4^3)\notag\\
 &\le C\norm\Phi_{\HS(\U,H^2)}V^{3/4}.
\end{align}
Substituting \eqref{eq:gradient} into the formula for $\mathrm DG$ yields
\[
 \begin{aligned}
 \sum_j|\mathrm DG(v)[\phi_j]|^2
 &=\frac9{49}V^{-8/7}\sum_j|\mathrm D\mathcal J(v)[\phi_j]|^2\\
 &\le C\norm\Phi_{\HS(\U,H^2)}^2V^{5/14}
 =C\norm\Phi_{\HS(\U,H^2)}^2G^{5/6}.
 \end{aligned}
\]
\end{proof}

\subsection{Comparison solutions and their properties}
Assume that $u_0\in H^{9/2}_0$ is deterministic and $\Phi\in\HS(\U,H^{9/2}_0)$. For an integer $K\ge1$, let $P_K$ denote the Fourier projection onto the modes $0<|k|\le K$, and consider
\begin{equation}\label{eq:smooth}
 \dd u_K+\left(\partial_x^3u_K+u_K\partial_x u_K+\lambda u_K\right)\dd t=P_K\Phi\dd W,
 \qquad u_K(0)=P_Ku_0.
\end{equation}
The nonlinear term is not projected, so $u_K$ need not have Fourier support in $|k|\le K$.

Denote by $\langle D\rangle^q$ the operator $(I-\partial_x^2)^{q/2}$. We assert the following estimates on $\T$:
\begin{gather}
 \norm{fg}_{H^q}\le C_{q,q_0}\bigl(\norm f_{H^q}\norm g_{H^{q_0}}+
               \norm f_{H^{q_0}}\norm g_{H^q}\bigr),
      \qquad q\ge0,\quad q_0>1/2,\label{eq:product}\\
 |\ip{\langle D\rangle^q(v\partial_x v)}{\langle D\rangle^qv}|\le C_q\norm{\partial_x v}_\infty\norm v_{H^q}^2,
      \qquad q>3/2.\label{eq:tame}
\end{gather}
The product estimate \eqref{eq:product} is standard. To obtain \eqref{eq:tame}, we use the one-dimensional specialization of the Kato--Ponce estimate in \cite[Lemma 2.2]{LinaresEtAl2018} (see also \cite{KatoPonce1988}):
\begin{equation}\label{eq:periodic-commutator}
 \norm{[\langle D\rangle^q,f]g}_2
 \le C_q\left[\norm f_{H^q}\norm g_\infty
       +(\norm f_\infty+\norm{\partial_x f}_\infty)\norm g_{H^{q-1}}\right],
\end{equation}
for smooth periodic functions $f,g$ and $q\ge1$, where $[\langle D\rangle^q,f]g=\langle D\rangle^q(fg)-f\langle D\rangle^qg$.

For smooth $v$, we apply \eqref{eq:periodic-commutator} with $f=v-\int_\T v$ and $g=\partial_x v$ and, since constants commute with $\langle D\rangle^q$, use the bounds
\[
 \norm{v-\int_\T v}_\infty\le2\pi\norm{\partial_x v}_\infty,\qquad
 \norm{v-\int_\T v}_{H^q}\le\norm v_{H^q},\qquad
 \norm{\partial_x v}_{H^{q-1}}\le\norm v_{H^q},
\]
to obtain $\norm{[\langle D\rangle^q,v]\partial_x v}_2\le C_q\norm{\partial_x v}_\infty\norm v_{H^q}$.
Using $[\langle D\rangle^q,\partial_x]=0$ and periodic integration by parts gives
\begin{equation}\label{eq:transport-pairing}
 \begin{aligned}
 \ip{\langle D\rangle^q(v\partial_x v)}{\langle D\rangle^qv}
 &=\ip{v\partial_x\langle D\rangle^qv}{\langle D\rangle^qv}
   +\ip{[\langle D\rangle^q,v]\partial_x v}{\langle D\rangle^qv}\\
 &=-\frac12\int_\T (\partial_x v)|\langle D\rangle^qv|^2
   +\ip{[\langle D\rangle^q,v]\partial_x v}{\langle D\rangle^qv}.
 \end{aligned}
\end{equation}
The Cauchy--Schwarz inequality now proves \eqref{eq:tame}. For $q>3/2$, a density argument using $H^q(\T)\hookrightarrow W^{1,\infty}(\T)$ extends the estimate to $v\in H^q(\T)$, with the pairing in \eqref{eq:tame} given by the last line of \eqref{eq:transport-pairing}.

\begin{proposition}[Comparison solutions]\label{prop:smooth}
For each integer $K\ge1$ and $\lambda\ge0$, equation \eqref{eq:smooth} has a unique global adapted solution $u_K$ with $u_K\in C([0,T];H_0^m)$ almost surely for every finite $T>0$ and integer $m\ge3$.
\end{proposition}

The proof is given in Appendix~\ref{app:comparison-existence}.

\begin{lemma}[Exponential integrability and Sobolev moment bounds]\label{lem:comparison-moments}
Fix $T>0$ and $0\le\Lambda<\infty$. The comparison solutions $u_K$ of \eqref{eq:smooth} satisfy, for every $A>0$,
\begin{equation}\label{eq:derivative-exp}
 \sup_{K\ge1,\ 0\le\lambda\le\Lambda}
 \E\exp\left(A\int_0^T\norm{\partial_x u_K(t)}_\infty\dd t\right)<\infty.
\end{equation}
For every $1\le p<\infty$, there is a constant $C_p$ such that
\begin{align}
 \sup_{\substack{K\ge1\\0\le\lambda\le\Lambda}}\left\|1+\sup_{0\le t\le T}\norm{u_K(t)}_{H^{9/2}}\right\|_{L^p(\Omega)}&\le C_p,\label{eq:low}\\
 \sup_{0\le\lambda\le\Lambda}\left\|1+\sup_{0\le t\le T}\norm{u_K(t)}_{H^9}\right\|_{L^p(\Omega)}&\le C_pK^{9/2},\qquad K\ge1.\label{eq:high}
\end{align}
For each such $p$ and every $9/2<q<9$, there is a constant $C_{p,q}$ such that, for all $K\ge1$,
\begin{equation}\label{eq:intermediate}
 \sup_{0\le\lambda\le\Lambda}\left\|1+\sup_{0\le t\le T}\norm{u_K(t)}_{H^q}\right\|_{L^p(\Omega)}
 \le C_{p,q}K^{q-9/2}.
\end{equation}
The bound in \eqref{eq:derivative-exp} depends only on $A,T,\Lambda$, $\norm{u_0}_{H^2}$ and $\norm\Phi_{\HS(\U,H^2)}$. The constant $C_p$ depends only on $p,T,\Lambda$, $\norm{u_0}_{H^{9/2}}$ and $\norm\Phi_{\HS(\U,H^{9/2})}$, while $C_{p,q}$ may additionally depend on $q$. Neither depends on $K$ or $\lambda\in[0,\Lambda]$.
\end{lemma}
\begin{proof}
By Proposition~\ref{prop:smooth}, $u_K$ has continuous $H^5$-valued paths. Applying It\^o's formula and Lemma~\ref{lem:generator}, with $\norm{P_K\Phi}_{\HS(\U,H^2)}\le\norm\Phi_{\HS(\U,H^2)}$, shows that $X_t=G(u_K(t))\ge1$ satisfies
\[
 \dd X_t\le C(1+X_t)\dd t+\dd M_t,
 \qquad
 \dd[M]_t\le CX_t^{5/6}\dd t,
\]
where $M$ is a continuous local martingale and $C$ depends only on $\Lambda$ and $\norm\Phi_{\HS(\U,H^2)}$, independently of $K$ and $\lambda\in[0,\Lambda]$.

The exponential Lyapunov argument of Cox, Hutzenthaler, and Jentzen~\cite{CoxHutzenthalerJentzen2024}, based on It\^o's formula and localization, applies to the present real-valued It\^o process and gives, for every $A>0$,
\[
 \E\exp\left(A\int_0^T G(u_K(t))\dd t\right)
 \le C\exp\!\left(CG(P_Ku_0)\right)\le C.
\]
The last inequality follows from the upper bound in \eqref{eq:coercive}, the contraction property of $P_K$ on $H^2$, and the embedding $H^2\hookrightarrow L^4$. The constants are uniform in $K$ and $\lambda\in[0,\Lambda]$. Since $\norm{\partial_x v}_\infty\le CG(v)$ by Lemma~\ref{lem:derivative-control}, replacing $A$ by $CA$ proves \eqref{eq:derivative-exp}.

To estimate the Sobolev moments, fix $q\in\{9/2,9\}$ and, for $L>\norm{P_Ku_0}_{H^{q+3}}$, define
\[
 \tau_L^{(q+3)}=\inf\{t\ge0:\norm{u_K(t)}_{H^{q+3}}\ge L\}.
\]
On $[0,\tau_L^{(q+3)}]$, $u_K$ is bounded in $H^{q+3}$, so the drift is $H^q$-valued. It\^o's formula and periodic integration by parts give
\begin{equation}\label{eq:sobolev-ito}
 \begin{aligned}
 \dd\norm{u_K}_{H^q}^2
 ={}&-2\ip{\langle D\rangle^q(u_K\partial_x u_K)}{\langle D\rangle^qu_K}\dd t
       -2\lambda\norm{u_K}_{H^q}^2\dd t\\
 &+\norm{P_K\Phi}_{\HS(\U,H^q)}^2\dd t
       +2\ip{\langle D\rangle^qu_K}{\langle D\rangle^qP_K\Phi\dd W}.
 \end{aligned}
\end{equation}
Set $\mathcal E_{q,K}(t)=1+K^{-2(q-9/2)}\norm{u_K(t)}_{H^q}^2$.
Since $(1+k^2)^{q-9/2}\le C_qK^{2(q-9/2)}$ for $0<|k|\le K$, the Fourier definitions of the Sobolev and Hilbert--Schmidt norms give
\[
 K^{-2(q-9/2)}\norm{P_K\Phi}_{\HS(\U,H^q)}^2
 =K^{-2(q-9/2)}\sum_j\norm{P_K\Phi e_j}_{H^q}^2
 \le C_q\norm\Phi_{\HS(\U,H^{9/2})}^2.
\]
Similarly, $\mathcal E_{q,K}(0)\le C_q(1+\norm{u_0}_{H^{9/2}}^2)$.
Multiplying \eqref{eq:sobolev-ito} by $K^{-2(q-9/2)}$, using \eqref{eq:tame}, and discarding the nonpositive damping term gives, for $t<\tau_L^{(q+3)}$,
\begin{equation*}
 \begin{split}
 \dd\mathcal E_{q,K}
 \le{}&C_q\norm{\partial_x u_K}_\infty\mathcal E_{q,K}\dd t
       +K^{-2(q-9/2)}\norm{P_K\Phi}_{\HS(\U,H^q)}^2\dd t\\
      &+2K^{-2(q-9/2)}\ip{\langle D\rangle^qu_K}{\langle D\rangle^qP_K\Phi\dd W}.
 \end{split}
\end{equation*}
Since $\mathcal E_{q,K}\ge1$, It\^o's formula for $\log\mathcal E_{q,K}$, after discarding the nonpositive It\^o correction, gives
\begin{equation}\label{eq:log}
 \dd\log\mathcal E_{q,K}\le C_q\norm{\partial_x u_K}_\infty\dd t
 +C\dd t+\dd\mathcal N_{q,K},
 \qquad t<\tau_L^{(q+3)},
\end{equation}
where
\[
 \mathcal N_{q,K}(t)=2K^{-2(q-9/2)}\int_0^{t\wedge\tau_L^{(q+3)}}
       \frac{\ip{\langle D\rangle^qu_K(s)}{\langle D\rangle^qP_K\Phi\dd W(s)}}{\mathcal E_{q,K}(s)}.
\]
Set $\mathcal N^*_{q,K}=\sup_{0\le t\le T}|\mathcal N_{q,K}(t)|$.
Integrating \eqref{eq:log} up to $t\wedge\tau_L^{(q+3)}$ gives
\begin{equation}\label{eq:stopped-sobolev-bound}
 \sup_{0\le t\le T}\mathcal E_{q,K}(t\wedge\tau_L^{(q+3)})
 \le\mathcal E_{q,K}(0)\exp\left(C_q\int_0^T\norm{\partial_x u_K}_\infty\dd t
                  +CT+\mathcal N^*_{q,K}\right).
\end{equation}
The integral term has finite exponential moments by \eqref{eq:derivative-exp}. To control $\mathcal N^*_{q,K}$, we use $K^{-2(q-9/2)}\norm{u_K}_{H^q}^2=\mathcal E_{q,K}-1$ and the Cauchy--Schwarz inequality to obtain
\[
 {}[\mathcal N_{q,K}]_T
 \le4K^{-2(q-9/2)}\norm{P_K\Phi}_{\HS(\U,H^q)}^2
       \int_0^{T\wedge\tau_L^{(q+3)}}\frac{\mathcal E_{q,K}(t)-1}{\mathcal E_{q,K}(t)^2}\dd t
 \le CT.
\]
Thus $\mathcal N_{q,K}$ is a square-integrable continuous martingale. Applying the exponential martingale inequality to $\pm\mathcal N_{q,K}$ gives
\[
 \PP(\mathcal N^*_{q,K}>z)\le2\exp\left(-\frac{z^2}{2CT}\right),\qquad z>0.
\]
For every $\alpha>0$, Tonelli's theorem therefore yields
\begin{equation}\label{eq:martingale-sup-exp}
 \begin{aligned}
 \E e^{\alpha\mathcal N^*_{q,K}}
 &=1+\alpha\int_0^\infty e^{\alpha z}\PP(\mathcal N^*_{q,K}>z)\dd z\\
 &\le1+2\alpha\int_0^\infty
       \exp\left(\alpha z-\frac{z^2}{2CT}\right)\dd z
 \le C_\alpha.
 \end{aligned}
\end{equation}

By \eqref{eq:stopped-sobolev-bound} and H\"older's inequality, for every $1\le p<\infty$,
\[
 \begin{aligned}
 &\E\left(\sup_{0\le t\le T}\mathcal E_{q,K}(t\wedge\tau_L^{(q+3)})\right)^{p/2}\\
 &\quad\le\mathcal E_{q,K}(0)^{p/2}e^{CpT/2}
 \E\exp\left(\frac{pC_q}{2}\int_0^T\norm{\partial_x u_K}_\infty\dd t
                  +\frac p2\mathcal N^*_{q,K}\right)\\
 &\quad\le\mathcal E_{q,K}(0)^{p/2}e^{CpT/2}
 \left(\E e^{pC_q\int_0^T\norm{\partial_x u_K}_\infty\dd t}\right)^{1/2}
 \left(\E e^{p\mathcal N^*_{q,K}}\right)^{1/2}
 \le C_p.
 \end{aligned}
\]
The last inequality uses the uniform bound on $\mathcal E_{q,K}(0)$,
\eqref{eq:derivative-exp} and \eqref{eq:martingale-sup-exp} with $\alpha=p$.
The constants $C$, $C_\alpha$ and $C_p$ are independent of $L$, $K$ and $\lambda\in[0,\Lambda]$.

By Proposition~\ref{prop:smooth}, $\sup_{0\le t\le T}\norm{u_K(t)}_{H^{q+3}}<\infty$ almost surely, so $\tau_L^{(q+3)}>T$ for all sufficiently large $L$. Fatou's lemma gives
\[
 \E\left(\sup_{0\le t\le T}\mathcal E_{q,K}(t)\right)^{p/2}
 \le\liminf_{L\to\infty}
       \E\left(\sup_{0\le t\le T}\mathcal E_{q,K}(t\wedge\tau_L^{(q+3)})\right)^{p/2}
 \le C_p.
\]
Therefore
\[
 \left\|1+\sup_{0\le t\le T}\norm{u_K(t)}_{H^q}\right\|_{L^p(\Omega)}
 \le1+K^{q-9/2}
       \left\|\left(\sup_{0\le t\le T}\mathcal E_{q,K}(t)\right)^{1/2}\right\|_{L^p(\Omega)}
 \le C_pK^{q-9/2},
\]
which proves \eqref{eq:low} and \eqref{eq:high}.

For $9/2<q<9$, set $\vartheta=(2q-9)/9$. The Sobolev interpolation inequality, H\"older's inequality and \eqref{eq:low}--\eqref{eq:high} give
\begin{align*}
 \left\|\sup_{0\le t\le T}\norm{u_K(t)}_{H^q}\right\|_{L^p(\Omega)}
 &\le
 \left\|\sup_{0\le t\le T}\norm{u_K(t)}_{H^{9/2}}\right\|_{L^p(\Omega)}^{1-\vartheta}
 \left\|\sup_{0\le t\le T}\norm{u_K(t)}_{H^9}\right\|_{L^p(\Omega)}^{\vartheta}\\*
 &\le C_{p,q}K^{(9/2)\vartheta}=C_{p,q}K^{q-9/2}.
\end{align*}
Since $K\ge1$, this also gives \eqref{eq:intermediate}.
\end{proof}

\subsection{Properties of the exact solution}

\begin{theorem}[Exponential integrability and comparison estimate]\label{thm:exact}
Fix $T>0$ and $0\le\Lambda<\infty$. Let $u_0\in H_0^{9/2}$ be deterministic and $\Phi\in\HS(\U,H_0^{9/2})$. For each $\lambda\in[0,\Lambda]$, equation \eqref{eq:spde} admits an adapted solution $u$ satisfying
\[
 u\in C([0,T];H_0^s)\qquad\text{for every }0\le s<9/2,\quad\text{a.s.}
\]
Pathwise uniqueness holds in $C([0,T];H_0^s)$ for every $3<s<9/2$.

Moreover, for every $A>0$,
\begin{equation}\label{eq:original-derivative-exp}
 \sup_{0\le\lambda\le\Lambda}
 \E\exp\left(A\int_0^T\norm{\partial_x u(t)}_\infty\dd t\right)<\infty.
\end{equation}
For every $1\le p<\infty$,
\begin{equation}\label{eq:original-sobolev-moment}
 \sup_{0\le\lambda\le\Lambda}
 \E\sup_{0\le t\le T}\norm{u(t)}_{H^{9/2}}^p<\infty.
\end{equation}
For every $r\ge1$, integer $K\ge1$ and $\lambda\in[0,\Lambda]$, the comparison solution $u_K$ of \eqref{eq:smooth}, with the same damping coefficient as $u$, satisfies
\begin{equation}\label{eq:approx}
 \left\|\sup_{0\le t\le T}\norm{u(t)-u_K(t)}_2\right\|_{L^{2r}(\Omega)}
 \le C_rK^{-9/2}.
\end{equation}
The bounds in \eqref{eq:original-derivative-exp} and \eqref{eq:original-sobolev-moment} and the constant $C_r$ in \eqref{eq:approx} depend only on $T,\Lambda$, $\norm{u_0}_{H^{9/2}}$, $\norm\Phi_{\HS(\U,H^{9/2})}$, and the respective parameters $A$, $p$ and $r$.
\end{theorem}
\begin{proof}
\medskip\noindent\emph{Step 1: comparison of $u_M$ and $u_K$.}
Fix $\lambda\in[0,\Lambda]$. For $M\ge K$, let $z=u_M-u_K$. Subtracting \eqref{eq:smooth} for $u_M$ and $u_K$ gives
\[
 \dd z+\bigl(\partial_x^3z+\partial_x(u_Kz)+z\partial_x z+\lambda z\bigr)\dd t
 =(P_M-P_K)\Phi\dd W.
\]
Periodic integration by parts gives
\[
 2\ip{\partial_x(u_Kz)}z=\int_\T (\partial_x u_K)z^2,\qquad \ip{z\partial_x z}z=0.
\]
It\^o's formula and the skew-adjointness of $\partial_x^3$ therefore yield
\begin{align}\label{eq:exact-difference}
 \dd\norm z_2^2+2\lambda\norm z_2^2\dd t
 ={}&-\int_\T (\partial_x u_K)z^2\dd t
  +\norm{(P_M-P_K)\Phi}_{\HS(\U,H)}^2\dd t\notag\\
 &+2\ip z{(P_M-P_K)\Phi\dd W}.
\end{align}
Since $(1+k^2)^{-9/2}\le K^{-9}$ for $|k|>K$, we have
\[
 \norm{(P_M-P_K)u_0}_2
 +\norm{(P_M-P_K)\Phi}_{\HS(\U,H)}\le CK^{-9/2}.
\]
Set
\[
 \mathcal E(t)=\sup_{0\le\tau\le t}\norm{z(\tau)}_2^2,\qquad
 \mathcal M(t)=2\int_0^t\ip{z(\tau)}{(P_M-P_K)\Phi\dd W(\tau)},
\]
and $\mathcal M_T^*=\sup_{0\le t\le T}|\mathcal M(t)|$. Integrating \eqref{eq:exact-difference}, discarding the nonnegative damping term, and taking the supremum over $[0,t]$ gives
\[
 \mathcal E(t)\le CK^{-9}+\mathcal M_T^*
        +\int_0^t\norm{\partial_x u_K(\tau)}_\infty\mathcal E(\tau)\dd\tau,
 \qquad 0\le t\le T.
\]
Gronwall's inequality implies
\begin{equation}\label{eq:exact-gronwall}
 \mathcal E(T)\le e^{\int_0^T\norm{\partial_x u_K}_\infty\dd t}
       (CK^{-9}+\mathcal M_T^*).
\end{equation}
The quadratic variation satisfies
\[
 \begin{aligned}
 {}[\mathcal M]_T
 &=4\int_0^T\norm{\big((P_M-P_K)\Phi\big)^*z(\tau)}_\U^2\dd\tau\\
 &\le4T\norm{(P_M-P_K)\Phi}_{\HS(\U,H)}^2\mathcal E(T)
 \le C_TK^{-9}\mathcal E(T).
 \end{aligned}
\]
The bound
\[
 \mathcal E(T)\le2\sup_{0\le t\le T}\norm{u_M(t)}_2^2
                  +2\sup_{0\le t\le T}\norm{u_K(t)}_2^2
\]
and \eqref{eq:low} give $\mathcal E(T)\in L^p(\Omega)$ for every $1\le p<\infty$. For $r\ge1$, the BDG inequality gives
\begin{equation}\label{eq:exact-bdg}
 \begin{aligned}
 &\norm{\mathcal M_T^*}_{L^{4r/3}(\Omega)}
 \le C\norm{[\mathcal M]_T^{1/2}}_{L^{4r/3}(\Omega)}
 \le CK^{-9/2}\norm{\mathcal E(T)^{1/2}}_{L^{4r/3}(\Omega)}\\
 &\le CK^{-9/2}\norm{\mathcal E(T)^{1/2}}_{L^{2r}(\Omega)}
 =CK^{-9/2}\norm{\mathcal E(T)}_{L^r(\Omega)}^{1/2},
 \end{aligned}
\end{equation}
where we used $4r/3\le2r$. Applying H\"older's inequality to \eqref{eq:exact-gronwall}, with \eqref{eq:derivative-exp} and \eqref{eq:exact-bdg}, yields
\[
 \begin{aligned}
 \norm{\mathcal E(T)}_{L^r(\Omega)}
 &\le CK^{-9}\Bigl\|e^{\int_0^T\norm{\partial_x u_K}_\infty\dd t}\Bigr\|_{L^r(\Omega)}
 +\Bigl\|e^{\int_0^T\norm{\partial_x u_K}_\infty\dd t}\Bigr\|_{L^{4r}(\Omega)}
       \norm{\mathcal M_T^*}_{L^{4r/3}(\Omega)}\\
 &\le CK^{-9}+CK^{-9/2}\norm{\mathcal E(T)}_{L^r(\Omega)}^{1/2}.
 \end{aligned}
\]
Since $\norm{\mathcal E(T)}_{L^r(\Omega)}<\infty$, Young's inequality gives
\[
 \left\|\sup_{0\le t\le T}\norm{u_M(t)-u_K(t)}_2\right\|_{L^{2r}(\Omega)}
 =\norm{\mathcal E(T)}_{L^r(\Omega)}^{1/2}\le CK^{-9/2}.
\]

\medskip\noindent\emph{Step 2: existence and pathwise uniqueness.}
For $0\le s<9/2$, the Sobolev interpolation inequality gives $\norm z_{H^s}\le\norm z_2^{1-2s/9}\norm z_{H^{9/2}}^{2s/9}$. Step~1, H\"older's inequality and \eqref{eq:low} yield
\[
 \left\|\sup_{0\le t\le T}\norm{u_M(t)-u_K(t)}_{H^s}\right\|_{L^p(\Omega)}
 \le C_{p,s}K^{s-9/2},\qquad 1\le p<\infty.
\]
These estimates give convergence in $L^p(\Omega;C([0,T];H^s))$ for every $1\le p<\infty$ and $0\le s<9/2$. The limits agree in $C([0,T];L^2)$ and define an adapted mean-zero process $u$. A diagonal argument over exponents increasing to $9/2$ yields a subsequence $(u_{M_j})$ converging almost surely to $u$ in $C([0,T];H^s)$ for every $0\le s<9/2$.

Fix $3<s<9/2$. For $v,w\in H^s$, using $v\partial_x v-w\partial_x w=\tfrac12\partial_x((v+w)(v-w))$, \eqref{eq:product}, and the boundedness of $\partial_x^3:H^s\to H^{s-3}$, we obtain
\begin{equation}\label{eq:drift-lipschitz}
 \begin{aligned}
 &\norm{\partial_x^3(v-w)+v\partial_x v-w\partial_x w+\lambda(v-w)}_{H^{s-3}}\\
 &\qquad\le C_{\Lambda,s}(1+\norm v_{H^s}+\norm w_{H^s})\norm{v-w}_{H^s}.
 \end{aligned}
\end{equation}
By \eqref{eq:drift-lipschitz}, H\"older's inequality, and convergence in $L^{2p}(\Omega;C([0,T];H^s))$, the drift integrals converge in $L^p(\Omega;C([0,T];H^{s-3}))$.
Since $P_K\to I$ strongly on $H^{9/2}_0$ and $\norm{P_K}_{H^{9/2}\to H^{9/2}}\le1$, dominated convergence in the Hilbert--Schmidt sum gives $\norm{(P_K-I)\Phi}_{\HS(\U,H^{9/2})}\to0$. The BDG inequality then yields
\[
 \left\|\sup_{0\le t\le T}
 \left\|\int_0^t(P_K-I)\Phi\dd W\right\|_{H^{9/2}}\right\|_{L^p(\Omega)}
 \le C_pT^{1/2}\norm{(P_K-I)\Phi}_{\HS(\U,H^{9/2})}
 \longrightarrow0.
\]
Since $P_Ku_0\to u_0$ in $H^{9/2}$, passing to the limit in the integral form of \eqref{eq:smooth} gives the following identity in $H^{s-3}$, almost surely for all $t\in[0,T]$:
\[
 u(t)=u_0-\int_0^t\bigl[\partial_x^3u(\tau)+u(\tau)\partial_x u(\tau)+\lambda u(\tau)\bigr]\dd\tau
       +\int_0^t\Phi\dd W(\tau).
\]
For pathwise uniqueness, let $v,w\in C([0,T];H_0^s)$ be two adapted solutions with the same initial data and Wiener process. Subtracting their integral equations and repeating the $L^2$ calculation in \eqref{eq:exact-difference} gives
\[
 \norm{v(t)-w(t)}_2^2
 \le\int_0^t\norm{\partial_x w(\tau)}_\infty\norm{v(\tau)-w(\tau)}_2^2\dd\tau.
\]
Since $\partial_x w\in L^1(0,T;L^\infty)$ a.s., Gronwall's inequality gives, almost surely, $v(t)=w(t)$ for all $t\in[0,T]$.

\medskip\noindent\emph{Step 3: exponential integrability, Sobolev moment bounds, and comparison estimate.}
For the subsequence from Step~2 and $3/2<s<9/2$, the embedding $H^s\hookrightarrow W^{1,\infty}$ gives
\[
 \left|\int_0^T\norm{\partial_x u_{M_j}}_\infty\dd t-
             \int_0^T\norm{\partial_x u}_\infty\dd t\right|
 \le C_sT\sup_{0\le t\le T}\norm{u_{M_j}(t)-u(t)}_{H^s}
 \longrightarrow0\qquad\text{a.s.}
\]
Applying Fatou's lemma for each $\lambda\in[0,\Lambda]$ and using \eqref{eq:derivative-exp}, we obtain
\[
 \sup_{0\le\lambda\le\Lambda}
 \E\exp\left(A\int_0^T\norm{\partial_x u}_\infty\dd t\right)
 \le\sup_{\substack{K\ge1\\0\le\lambda\le\Lambda}}
      \E\exp\left(A\int_0^T\norm{\partial_x u_K}_\infty\dd t\right)<\infty,
\]
which proves \eqref{eq:original-derivative-exp}.

To estimate the $H^{9/2}$ moments, fix an integer $\ell\ge1$. Boundedness of $P_\ell:L^2\to H^{9/2}$ gives $P_\ell u_{M_j}\to P_\ell u$ almost surely in $C([0,T];H^{9/2})$. Fatou's lemma and \eqref{eq:low} yield, for $1\le p<\infty$,
\[
 \E\sup_{0\le t\le T}\norm{P_\ell u(t)}_{H^{9/2}}^p
 \le\liminf_{j\to\infty}\E\sup_{0\le t\le T}\norm{u_{M_j}(t)}_{H^{9/2}}^p\le C_p.
\]
The partial Fourier norms increase with $\ell$, and their suprema in time increase to $\sup_{0\le t\le T}\norm{u(t)}_{H^{9/2}}$. Monotone convergence therefore gives
\[
 \E\sup_{0\le t\le T}\norm{u(t)}_{H^{9/2}}^p
 =\lim_{\ell\to\infty}\E\sup_{0\le t\le T}\norm{P_\ell u(t)}_{H^{9/2}}^p
 \le C_p,
\]
uniformly for $\lambda\in[0,\Lambda]$, proving \eqref{eq:original-sobolev-moment}.

For each fixed $K$, $M_j\ge K$ for all sufficiently large $j$. Fatou's lemma and Step~1 give
\[
 \E\sup_{0\le t\le T}\norm{u(t)-u_K(t)}_2^{2r}
 \le\liminf_{j\to\infty}
      \E\sup_{0\le t\le T}\norm{u_{M_j}(t)-u_K(t)}_2^{2r}
 \le C_rK^{-9r}.
\]
Taking the $2r$-th root proves \eqref{eq:approx}.
\end{proof}

\section{Strong error estimates}\label{sec:strong}

Under the assumptions of Theorem~\ref{thm:exact}, fix an integer $K\ge1$ and let $(U^n,m^n)$ be any solution of \eqref{eq:scheme}. Let $L=\partial_x^3+\lambda I$ and write $F(v)=Lv+v\partial_x v$. Recall that $Q=(I+\tfrac h2L)^{-1}$. Applying $Q$ to \eqref{eq:scheme}, with $2m^n=U^{n+1}+U^n$ and $\tfrac h2QL=I-Q$, gives
\[
 U^{n+1}=(2Q-I)U^n-hQ\bigl(m^n\partial_x m^n\bigr)+Q\Phi\Delta W^n.
\]
Set $R=2Q-I$. For every $h>0$, $\lambda\ge0$ and $q\in\mathbb R$,
\begin{equation}\label{eq:resolvents}
 \begin{gathered}
 \norm{Rv}_2^2+2\lambda h\norm{Qv}_2^2=\norm v_2^2,
 \qquad v\in L^2,\\
 \norm R_{H^q\to H^q}\le1,\qquad
 \norm{R^*}_{H^q\to H^q}\le1.
 \end{gathered}
\end{equation}
Indeed, $Qv\in H^3$, $Rv-v=-hLQv$ and $Rv+v=2Qv$. Periodic integration by parts gives
\[
 \norm{Rv}_2^2-\norm v_2^2
 =\ip{Rv-v}{Rv+v}
 =-2h\ip{LQv}{Qv}
 =-2\lambda h\norm{Qv}_2^2.
\]
The multiplier bounds give the contractions in \eqref{eq:resolvents}. The energy identity is the discrete analogue of $\frac{\mathrm d}{\mathrm d t}\norm{v(t)}_2^2=-2\lambda\norm{v(t)}_2^2$ for $\partial_t v+Lv=0$, with midpoint value $Qv=(v+Rv)/2$.

For $t_n\le T$, set $e^n=u_K(t_n)-U^n$. For $t_{n+1}\le T$, define
\begin{equation}\label{eq:local-defs}
 \begin{gathered}
 a^n=\tfrac12\bigl(u_K(t_n)+u_K(t_{n+1})\bigr),\qquad
 \kappa^n=ha^n-\int_{t_n}^{t_{n+1}}u_K(t)\dd t,\\
 \bar e^n=\tfrac12(e^{n+1}+e^n)=a^n-m^n.
 \end{gathered}
\end{equation}

\begin{proposition}[Local error estimates and an $L^2$ error propagation identity]\label{prop:one-step}
For $0<h\le1$, $\lambda\in[0,\Lambda]$ and $t_{n+1}\le T$, there exist $\F_{t_{n+1}}$-measurable random variables $M^n\in H^2$ and $D^n\in L^2$ such that $\E(M^n\mid\F_{t_n})=0$ and
\begin{equation}\label{eq:error-equation}
 e^{n+1}-e^n+\frac h2L(e^{n+1}+e^n)
 +h\bigl(a^n\partial_x a^n-m^n\partial_x m^n\bigr)=M^n+D^n.
\end{equation}
For every $p\ge2$, the following local error estimates hold:
\begin{align}
 \norm{M^n}_{L^p(\Omega;L^2)}
 &\le C_p(h^{3/2}+h^{5/2}K^{3/2}+h^{1/2}K^{-9/2}),\label{eq:localM0}\\
 \norm{M^n}_{L^p(\Omega;H^2)}
 &\le C_p(h^{3/2}K^{1/2}+h^{5/2}K^{7/2}+h^{1/2}K^{-5/2}),\label{eq:localM2}\\
 \norm{D^n}_{L^p(\Omega;L^2)}
 &\le C_p(h^2+h^3K^{9/2}),\label{eq:localD}\\
 \norm{\kappa^n}_{L^p(\Omega;H^2)}
 &\le C_p(h^{3/2}+h^2K^{1/2}).\label{eq:localq}
\end{align}
With $w^n=R^*QM^n$, the errors satisfy the following identity:
\begin{align}\label{eq:error-identity}
 &\norm{e^{n+1}}_2^2-\norm{e^n}_2^2
   +h\int_\T(\partial_x a^n)(\bar e^n)^2
   +2\lambda h\norm{\bar e^n-\tfrac12QM^n}_2^2\notag\\
 ={}&2\ip{D^n}{\bar e^n}+\norm{QM^n}_2^2
        +2\ip{e^n}{w^n}
        -h\ip{a^n\partial_x a^n-m^n\partial_x m^n}{w^n}
        +\ip{D^n}{w^n}.
\end{align}
The constants $C_p$ depend only on $p,T,\Lambda$, $\norm{u_0}_{H^{9/2}}$ and $\norm\Phi_{\HS(\U,H^{9/2})}$.
\end{proposition}
\begin{proof}
\medskip\noindent\emph{Step 1: error decomposition.}
For $0\le\tau\le h$, set $v_\tau=u_K(t_n+\tau)$ and $\chi_\tau=v_\tau-a^n$.
Since $\int_0^h\chi_\tau\dd\tau=-\kappa^n$, the expansion
$v_\tau \partial_x v_\tau=a^n\partial_x a^n+\partial_x(a^n\chi_\tau)+\chi_\tau\partial_x \chi_\tau$
and the integral form of \eqref{eq:smooth} give
\[
 \begin{aligned}
 &v_h-v_0+hF(a^n)-\Phi\Delta W^n\\
 ={}&\underbrace{L\kappa^n}_{\mathrm I}
       +\underbrace{\partial_x(a^n\kappa^n)}_{\mathrm{II}}
       +\underbrace{\left(-\int_0^h\chi_\tau\partial_x \chi_\tau\dd\tau\right)}_{\mathrm{III}}
       +\underbrace{\bigl(-(I-P_K)\Phi\Delta W^n\bigr)}_{\mathrm{IV}}.
 \end{aligned}
\]

For terms $\mathrm I$ and $\mathrm{II}$, the product rule with $\tau-h/2$ gives
\[
 \kappa^n=\frac h2(v_h+v_0)-\int_0^h v_\tau\dd\tau
   =\int_0^h(\tau-h/2)\dd v_\tau.
\]
Substituting $\dd v_\tau=-F(v_\tau)\dd\tau+P_K\Phi\dd W(t_n+\tau)$ gives $\kappa^n=S_W+S_L+S_{\mathrm{nl}}$, where
\[
 \begin{aligned}
 S_W&=\int_0^h(\tau-h/2)P_K\Phi\dd W(t_n+\tau),\\
 S_L&=\int_0^h(h/2-\tau)Lv_\tau\dd\tau,\qquad
 S_{\mathrm{nl}}=\int_0^h(h/2-\tau)v_\tau \partial_x v_\tau\dd\tau.
 \end{aligned}
\]
Thus $\mathrm I=LS_W+LS_L+LS_{\mathrm{nl}}$. For $LS_L$, integration by parts with $g_h(\tau)=\tau(h-\tau)/2$ gives
\[
 \begin{aligned}
 LS_L&=\int_0^h g_h'(\tau)L^2v_\tau\dd\tau
      =-\int_0^h g_h(\tau)L^2\dd v_\tau\\
 &=\int_0^h g_h(\tau)L^2F(v_\tau)\dd\tau
   -\int_0^h g_h(\tau)L^2P_K\Phi\dd W(t_n+\tau).
 \end{aligned}
\]

For term $\mathrm{II}$, we write
\[
 \partial_x(a^n\kappa^n)=\partial_x(v_0S_W)+\partial_x((a^n-v_0)S_W)
 +\partial_x\bigl(a^n(S_L+S_{\mathrm{nl}})\bigr).
\]
Define $M^n$ and $D^n$ by
\begin{align}
 M^n={}&LS_W+\partial_x(v_0S_W)
       -\int_0^hg_h(\tau)L^2P_K\Phi\dd W(t_n+\tau)
       -(I-P_K)\Phi\Delta W^n,\label{eq:M}\\
 D^n={}&\int_0^hg_h(\tau)L^2F(v_\tau)\dd\tau+LS_{\mathrm{nl}}
       +\partial_x((a^n-v_0)S_W)\notag\\
 &+\partial_x(a^n(S_L+S_{\mathrm{nl}}))-\int_0^h\chi_\tau\partial_x \chi_\tau\dd\tau.\label{eq:D}
\end{align}
By \eqref{eq:M}--\eqref{eq:D}, $v_h-v_0+hF(a^n)=\Phi\Delta W^n+M^n+D^n$. Subtracting \eqref{eq:scheme} gives \eqref{eq:error-equation}.

All terms in \eqref{eq:M}--\eqref{eq:D} are $\F_{t_{n+1}}$-measurable.
For fixed $K$, Lemma~\ref{lem:comparison-moments} and Gaussian moment bounds
give their integrability in the stated spaces.
Since $v_0$ is $\F_{t_n}$-measurable and $S_W$ is independent of
$\F_{t_n}$ with mean zero,
\[
 \E\bigl(\partial_x(v_0S_W)\mid\F_{t_n}\bigr)
 =\partial_x\bigl(v_0\E(S_W\mid\F_{t_n})\bigr)=0.
\]
The remaining terms in \eqref{eq:M} are It\^o integrals with deterministic integrands over $[t_n,t_{n+1}]$ and therefore have zero conditional expectation given $\F_{t_n}$. Thus $\E(M^n\mid\F_{t_n})=0$.

\medskip\noindent\emph{Step 2: local moment estimates.}
We first prove \eqref{eq:localM0}--\eqref{eq:localM2} for $M^n$. A standard Fourier argument gives, for every Sobolev index $\beta$,
\[
 \begin{aligned}
 \norm{P_K\Phi}_{\HS(\U,H^\beta)}
 &\le C K^{\max\{\beta-9/2,0\}}\norm\Phi_{\HS(\U,H^{9/2})},\\
 \norm{(I-P_K)\Phi}_{\HS(\U,H^\beta)}
 &\le CK^{\beta-9/2}\norm\Phi_{\HS(\U,H^{9/2})},\qquad \beta\le9/2.
 \end{aligned}
\]
The deterministic time weights satisfy
\[
 \int_0^h(\tau-h/2)^2\dd\tau=h^3/12,
 \qquad \int_0^h g_h(\tau)^2\dd\tau=h^5/120.
\]
It\^o's isometry and Gaussian moment bounds, together with
$\norm{Lf}_{H^\beta}\le C_\Lambda\norm f_{H^{\beta+3}}$, give
\[
 \begin{aligned}
 \norm{S_W}_{L^p(\Omega;H^\beta)}
 &\le C_ph^{3/2}\norm{P_K\Phi}_{\HS(\U,H^\beta)},\\
 \norm{LS_W}_{L^p(\Omega;H^\beta)}
 &\le C_ph^{3/2}\norm{P_K\Phi}_{\HS(\U,H^{\beta+3})},\\
 \left\|\int_0^h g_h(\tau)L^2P_K\Phi\dd W(t_n+\tau)\right\|_{L^p(\Omega;H^\beta)}
 &\le C_ph^{5/2}\norm{P_K\Phi}_{\HS(\U,H^{\beta+6})}.
 \end{aligned}
\]
The $H^5$, $H^6$ and $H^8$ projection bounds yield the factors $K^{1/2}$, $K^{3/2}$ and $K^{7/2}$ in \eqref{eq:localM0}--\eqref{eq:localM2}. For the product term, \eqref{eq:product} gives $\norm{\partial_x(fg)}_{H^2}\le C\norm f_{H^3}\norm g_{H^3}$. Using independence and \eqref{eq:low} in this bound gives
\[
 \norm{\partial_x(v_0S_W)}_{L^p(\Omega;H^2)}
 \le C\norm{v_0}_{L^p(\Omega;H^3)}
          \norm{S_W}_{L^p(\Omega;H^3)}
 \le Ch^{3/2}.
\]
Finally, for $\beta\le9/2$, the last term in \eqref{eq:M} satisfies
\[
 \norm{(I-P_K)\Phi\Delta W^n}_{L^p(\Omega;H^\beta)}
 \le C_ph^{1/2}\norm{(I-P_K)\Phi}_{\HS(\U,H^\beta)}
 \le Ch^{1/2}K^{\beta-9/2}.
\]
Combining the estimates for the four terms in \eqref{eq:M} proves \eqref{eq:localM0}--\eqref{eq:localM2}.

For the first term in \eqref{eq:D}, use $L^2F(v)=L^3v+L^2(v\partial_x v)$ and
\[
 \begin{aligned}
 \norm{L^3v}_2&\le C_\Lambda\norm v_{H^9},\\
 \norm{L^2(v\partial_x v)}_2&\le C_\Lambda\norm{v\partial_x v}_{H^6}
 \le C_\Lambda\norm v_{H^{9/2}}\norm v_{H^7}.
 \end{aligned}
\]
The last inequality follows from \eqref{eq:product} with $q=6$ and $q_0=7/2$. Using these bounds, $\int_0^h g_h(\tau)\dd\tau=h^3/12$, Minkowski's and H\"older's inequalities, \eqref{eq:high}, and \eqref{eq:intermediate} with $q=7$, we obtain
\begin{equation}\label{eq:localD-leading-term}
 \begin{aligned}
 &\Biggl\|\int_0^h g_h(\tau)L^2F(v_\tau)\dd\tau\Biggr\|_{L^p(\Omega;L^2)}
 \le Ch^3\Bigl\|\sup_{0\le t\le T}\norm{u_K(t)}_{H^9}\Bigr\|_{L^p(\Omega)}\\
 &\quad+Ch^3\Bigl\|\sup_{0\le t\le T}\norm{u_K(t)}_{H^{9/2}}\Bigr\|_{L^{2p}(\Omega)}
  \Bigl\|\sup_{0\le t\le T}\norm{u_K(t)}_{H^7}\Bigr\|_{L^{2p}(\Omega)}
 \\
 &\quad\le Ch^3(K^{9/2}+K^{5/2})\le Ch^3K^{9/2}.
 \end{aligned}
\end{equation}
For the remaining terms in $D^n$, we first establish an estimate for time increments of $u_K$ in $H^1$. By \eqref{eq:product} and \eqref{eq:low},
\[
 \norm{F(u_K(t))}_{L^{2p}(\Omega;H^1)}
 \le C_\Lambda\left(\norm{u_K(t)}_{L^{2p}(\Omega;H^4)}
       +\norm{u_K(t)}_{L^{4p}(\Omega;H^2)}^2\right)\le C.
\]
Using $\norm{P_K\Phi}_{\HS(\U,H^1)}\le\norm\Phi_{\HS(\U,H^{9/2})}$, the integral form of \eqref{eq:smooth} gives, for $0\le\tau\le1$ and $0\le t\le t+\tau\le T$,
\[
 \begin{aligned}
 &\norm{u_K(t+\tau)-u_K(t)}_{L^{2p}(\Omega;H^1)}\\
 &\quad\le\int_t^{t+\tau}\norm{F(u_K(\sigma))}_{L^{2p}(\Omega;H^1)}\dd\sigma
       +C_p\tau^{1/2}\norm{P_K\Phi}_{\HS(\U,H^1)}\\
 &\quad\le C(\tau+\tau^{1/2})\le C\tau^{1/2}.
 \end{aligned}
\]
Since $a^n-v_0=(v_h-v_0)/2$ and
$\chi_\tau=(v_\tau-v_0)/2+(v_\tau-v_h)/2$, it follows that
\[
 \norm{a^n-v_0}_{L^{2p}(\Omega;H^1)}
 +\sup_{0\le\tau\le h}\norm{\chi_\tau}_{L^{2p}(\Omega;H^1)}\le Ch^{1/2}.
\]

Since $\int_0^h|h/2-\tau|\dd\tau=h^2/4$, the bounds
$\norm{Lf}_{H^1}\le C_\Lambda\norm f_{H^4}$, $\norm{L(v\partial_x v)}_2\le C_\Lambda\norm v_{H^4}^2$, and
$\norm{v\partial_x v}_{H^1}\le C\norm v_{H^2}^2$ imply
\[
 \begin{gathered}
 \norm{LS_{\mathrm{nl}}}_{L^p(\Omega;L^2)}
 \le Ch^2\left\|\sup_{0\le t\le T}\norm{u_K(t)}_{H^4}\right\|_{L^{2p}(\Omega)}^2
 \le Ch^2,\\
 \norm{S_L}_{L^{2p}(\Omega;H^1)}
 +\norm{S_{\mathrm{nl}}}_{L^{2p}(\Omega;H^1)}\le Ch^2.
 \end{gathered}
\]
Using $\norm{\partial_x(fg)}_2\le C\norm f_{H^1}\norm g_{H^1}$ and H\"older's inequality, we obtain
\[
 \begin{aligned}
 &\norm{\partial_x(a^n(S_L+S_{\mathrm{nl}}))}_{L^p(\Omega;L^2)}\\
 &\quad\le C\norm{a^n}_{L^{2p}(\Omega;H^1)}
       \bigl(\norm{S_L}_{L^{2p}(\Omega;H^1)}
            +\norm{S_{\mathrm{nl}}}_{L^{2p}(\Omega;H^1)}\bigr)
 \le Ch^2,\\
 &\norm{\partial_x((a^n-v_0)S_W)}_{L^p(\Omega;L^2)}\\
 &\quad\le C\norm{a^n-v_0}_{L^{2p}(\Omega;H^1)}
          \norm{S_W}_{L^{2p}(\Omega;H^1)}
 \le Ch^{1/2}h^{3/2}=Ch^2,\\
 &\Biggl\|\int_0^h\chi_\tau\partial_x \chi_\tau\dd\tau\Biggr\|_{L^p(\Omega;L^2)}
 \le C\int_0^h\norm{\chi_\tau}_{L^{2p}(\Omega;H^1)}^2\dd\tau
 \le Ch^2.
 \end{aligned}
\]
Combining these bounds with \eqref{eq:localD-leading-term} proves \eqref{eq:localD}.

Finally, to prove \eqref{eq:localq}, we return to $\kappa^n=S_W+S_L+S_{\mathrm{nl}}$. The bound $\norm{Lf}_{H^\beta}\le C_\Lambda\norm f_{H^{\beta+3}}$ and \eqref{eq:product} give
\[
 \norm{S_L}_{H^2}\le Ch^2\sup_{0\le\tau\le h}\norm{v_\tau}_{H^5},
 \qquad
 \norm{S_{\mathrm{nl}}}_{H^2}\le Ch^2\sup_{0\le\tau\le h}\norm{v_\tau}_{H^3}^2.
\]
The Gaussian estimate for $S_W$, \eqref{eq:intermediate} with $q=5$ and \eqref{eq:low} therefore yield
\[
 \norm{\kappa^n}_{L^p(\Omega;H^2)}
 \le C\bigl(h^{3/2}+h^2K^{1/2}+h^2\bigr)
 \le C\bigl(h^{3/2}+h^2K^{1/2}\bigr).
\]

\medskip\noindent\emph{Step 3: the error propagation identity.}
Propositions~\ref{prop:roots} and~\ref{prop:smooth} imply that $U^n,m^n,a^n,\bar e^n,e^n\in H^3$ almost surely. Moreover, \eqref{eq:M}--\eqref{eq:D} give $M^n\in H^2$ and $D^n\in L^2$, so all pairings below are well defined. Using $m^n=a^n-\bar e^n$ from \eqref{eq:local-defs}, we obtain $a^n\partial_x a^n-m^n\partial_x m^n=\partial_x(a^n\bar e^n)-\bar e^n\partial_x\bar e^n$. Periodic integration by parts gives
\begin{equation}\label{eq:nonlinear-pairing}
 2\ip{a^n\partial_x a^n-m^n\partial_x m^n}{\bar e^n}
 =-2\int_\T a^n\bar e^n\partial_x\bar e^n
 =\int_\T(\partial_x a^n)(\bar e^n)^2.
\end{equation}
We now write \eqref{eq:error-equation} in the operator form
\[
 \bigl(I+\tfrac h2L\bigr)e^{n+1}
 =\bigl(I-\tfrac h2L\bigr)e^n-h\bigl(a^n\partial_x a^n-m^n\partial_x m^n\bigr)
 +M^n+D^n.
\]
Applying $Q$ and using $Q(I-hL/2)=R$ gives
\[
 e^{n+1}=Re^n-hQ\bigl(a^n\partial_x a^n-m^n\partial_x m^n\bigr)
               +QM^n+QD^n.
\]
Using $I+R=2Q$, we rearrange this as
\[
 \begin{aligned}
 &e^{n+1}+\frac h2\bigl(a^n\partial_x a^n-m^n\partial_x m^n\bigr)
       -\frac12D^n-QM^n\\
 &\qquad=R\left(e^n-\frac h2\bigl(a^n\partial_x a^n-m^n\partial_x m^n\bigr)
                     +\frac12D^n\right),\\
 &\bar e^n-\frac12QM^n
   =Q\left(e^n-\frac h2\bigl(a^n\partial_x a^n-m^n\partial_x m^n\bigr)
                     +\frac12D^n\right).
 \end{aligned}
\]
Applying \eqref{eq:resolvents} to the common argument of $R$ and $Q$, expanding the squared norms, and using $\ip{Rv}{QM^n}=\ip{v}{w^n}$ together with \eqref{eq:nonlinear-pairing} gives \eqref{eq:error-identity}.
\end{proof}

\begin{theorem}[Strong error estimates]\label{thm:main}
Let $u_0\in H^{9/2}_0$ be deterministic and let $\Phi\in\HS(\U,H^{9/2}_0)$. For every $T>0$, $0\le\Lambda<\infty$ and real $r\ge1$, there exists $C<\infty$ such that, for every $0<h\le1$ and $\lambda\in[0,\Lambda]$, the solution $u$ of \eqref{eq:spde} given by Theorem~\ref{thm:exact} and every solution of \eqref{eq:scheme} satisfy
\begin{equation}\label{eq:main}
 \E\max_{0\le n\le\lfloor T/h\rfloor}
       \norm{u(t_n)-U^n}_2^{2r}\le Ch^{2r}.
\end{equation}
The constant $C$ depends only on $r,T,\Lambda$, $\norm{u_0}_{H^{9/2}}$ and $\norm\Phi_{\HS(\U,H^{9/2})}$.
\end{theorem}

\begin{proof}
Let $(U^n,m^n)_{n\ge0}$ be a solution and set
$N=\lfloor T/h\rfloor$, $K=\lceil h^{-2/9}\rceil$ and
$\mathcal E_j=\max_{0\le n\le j}\norm{e^n}_2^2$ for $0\le j\le N$.
By the triangle inequality and the comparison estimate \eqref{eq:approx},
\begin{equation}\label{eq:total-error-decomposition}
 \begin{aligned}
 &\left\|\max_{0\le n\le N}\norm{u(t_n)-U^n}_2\right\|_{L^{2r}(\Omega)}\\
 &\quad\le\left\|\sup_{0\le t\le T}\norm{u(t)-u_K(t)}_2\right\|_{L^{2r}(\Omega)}
       +\left\|\max_{0\le n\le N}\norm{e^n}_2\right\|_{L^{2r}(\Omega)}\\
 &\quad\le CK^{-9/2}+\norm{\mathcal E_N}_{L^r(\Omega)}^{1/2}
       \le Ch+\norm{\mathcal E_N}_{L^r(\Omega)}^{1/2}.
 \end{aligned}
\end{equation}
It therefore suffices to prove $\norm{\mathcal E_N}_{L^r(\Omega)}\le Ch^2$.

\medskip\noindent\emph{Step 1: error propagation inequality.}
For the nonlinear term in \eqref{eq:error-identity}, periodic integration by parts gives
\[
 \ip{a^n\partial_x a^n-m^n\partial_x m^n}{w^n}
 =-\int_\T a^n\bar e^n\partial_x w^n
       +\frac12\int_\T(\bar e^n)^2\partial_x w^n.
\]
The embedding $H^2\hookrightarrow W^{1,\infty}$ gives
\[
 \left|\ip{a^n\partial_x a^n-m^n\partial_x m^n}{w^n}\right|
 \le C\norm{w^n}_{H^2}
       \bigl(\norm{a^n}_2\norm{\bar e^n}_2+\norm{\bar e^n}_2^2\bigr).
\]
Young's inequality and $\norm{w^n}_2\le\norm{M^n}_2$ imply
\begin{align*}
 2|\ip{D^n}{\bar e^n}|
 &\le h\norm{\bar e^n}_2^2+h^{-1}\norm{D^n}_2^2,\\
 Ch\norm{w^n}_{H^2}\norm{a^n}_2\norm{\bar e^n}_2
 &\le h\norm{\bar e^n}_2^2+Ch\norm{w^n}_{H^2}^2\norm{a^n}_2^2,\\
 |\ip{D^n}{w^n}|
 &\le(2h)^{-1}\norm{D^n}_2^2+(h/2)\norm{M^n}_2^2.
\end{align*}
We discard the nonnegative damping term in \eqref{eq:error-identity} and use $|\int_\T(\partial_x a^n)(\bar e^n)^2|\le\norm{\partial_x a^n}_\infty\norm{\bar e^n}_2^2$ and $\norm{QM^n}_2\le\norm{M^n}_2$. Since $h\le1$ and $\norm{\bar e^n}_2^2\le\tfrac12(\norm{e^n}_2^2+\norm{e^{n+1}}_2^2)$, we obtain
\begin{multline}\label{eq:error-bound}
 \norm{e^{n+1}}_2^2-\norm{e^n}_2^2\le2\ip{e^n}{w^n}
 +C\bigl(\norm{M^n}_2^2+h\norm{w^n}_{H^2}^2\norm{a^n}_2^2
       +h^{-1}\norm{D^n}_2^2\bigr)\\
 +Ch\bigl(1+\norm{\partial_x a^n}_\infty+\norm{w^n}_{H^2}\bigr)
       \bigl(\norm{e^n}_2^2+\norm{e^{n+1}}_2^2\bigr).
\end{multline}
Here $C$ is independent of $h,K,n$ and $\lambda$.

Differentiating the identity $ha^n=\int_{t_n}^{t_{n+1}}u_K(t)\dd t+\kappa^n$ from \eqref{eq:local-defs} and using $H^2\hookrightarrow W^{1,\infty}$ gives
\[
 h\norm{\partial_x a^n}_\infty
 \le\int_{t_n}^{t_{n+1}}\norm{\partial_x u_K(t)}_\infty\dd t+C\norm{\kappa^n}_{H^2}.
\]
Consequently, \eqref{eq:error-bound} yields
\begin{equation}\label{eq:recursion}
 \norm{e^{n+1}}_2^2-\norm{e^n}_2^2
 \le2\ip{e^n}{w^n}
 +\gamma_n\bigl(\norm{e^n}_2^2+\norm{e^{n+1}}_2^2\bigr)
       +B_n,
\end{equation}
where
\begin{align*}
 \gamma_n&=C_*\left[h+\int_{t_n}^{t_{n+1}}\norm{\partial_x u_K(t)}_\infty\dd t
                  +\norm{\kappa^n}_{H^2}+h\norm{w^n}_{H^2}\right],\\
 B_n&=C_*\left[\norm{M^n}_2^2
          +h\norm{w^n}_{H^2}^2\norm{a^n}_2^2
          +h^{-1}\norm{D^n}_2^2\right].
\end{align*}
Here $C_*>0$ is chosen sufficiently large, independently of $h,K,n$ and $\lambda$.

\medskip\noindent\emph{Step 2: moment and martingale estimates.}
Since $h^{-2/9}\le K\le2h^{-2/9}$, it follows from Proposition~\ref{prop:one-step}, \eqref{eq:Q-contraction} and \eqref{eq:resolvents} that, for every $p\ge2$,
\begin{equation}\label{eq:chosen-local}
 \begin{aligned}
 \norm{M^n}_{L^p(\Omega;L^2)}&\le Ch^{3/2},&
 \norm{w^n}_{L^p(\Omega;H^2)}&\le Ch^{19/18},\\
 \norm{D^n}_{L^p(\Omega;L^2)}&\le Ch^2,&
 \norm{\kappa^n}_{L^p(\Omega;H^2)}&\le Ch^{3/2}.
 \end{aligned}
\end{equation}
Indeed, $h^{5/2}K^{3/2}\le Ch^{13/6}$, $h^{1/2}K^{-9/2}\le h^{3/2}$, $h^3K^{9/2}\le Ch^2$ and $h^2K^{1/2}\le Ch^{17/9}$,
while the three terms in \eqref{eq:localM2} give
\[
 \norm{w^n}_{L^p(\Omega;H^2)}
 \le C\bigl(h^{25/18}+h^{31/18}+h^{19/18}\bigr)
 \le Ch^{19/18}.
\]

We next estimate the sums over the time steps:
\begin{equation}\label{eq:summed}
 \left\|\sum_{0\le n<N}\norm{w^n}_2^2\right\|_{L^{2r}(\Omega)}\le Ch^2,
 \qquad
 \left\|\sum_{0\le n<N}B_n\right\|_{L^{4r/3}(\Omega)}\le Ch^2.
\end{equation}
Since $\norm{w^n}_2\le\norm{M^n}_2$, we use Minkowski's inequality and \eqref{eq:chosen-local} to obtain
\[
 \left\|\sum_{0\le n<N}\norm{w^n}_2^2\right\|_{L^{2r}(\Omega)}
 \le\sum_{0\le n<N}\norm{w^n}_{L^{4r}(\Omega;L^2)}^2
 \le CNh^3\le Ch^2.
\]
To estimate $\sum_{0\le n<N}B_n$, we use its definition and Minkowski's inequality:
\[
 \begin{aligned}
 &\Bigl\|\sum_{0\le n<N}B_n\Bigr\|_{L^{4r/3}(\Omega)}
 \le C_*\sum_{0\le n<N}\norm{M^n}_{L^{8r/3}(\Omega;L^2)}^2\\
 &\quad+C_*h\sum_{0\le n<N}
 \bigl\|\norm{w^n}_{H^2}^2\norm{a^n}_2^2\bigr\|_{L^{4r/3}(\Omega)}
 +C_*h^{-1}\sum_{0\le n<N}\norm{D^n}_{L^{8r/3}(\Omega;L^2)}^2.
 \end{aligned}
\]
The bounds for $M^n$ and $D^n$ in \eqref{eq:chosen-local} give
\[
 \sum_{0\le n<N}\norm{M^n}_{L^{8r/3}(\Omega;L^2)}^2
 +h^{-1}\sum_{0\le n<N}\norm{D^n}_{L^{8r/3}(\Omega;L^2)}^2
 \le CNh^3\le Ch^2.
\]
For the sum involving $w^n$ and $a^n$, the inequality $\norm{a^n}_2\le\sup_{0\le t\le T}\norm{u_K(t)}_2$ and \eqref{eq:low} give uniform moment bounds for $a^n$. H\"older's inequality and \eqref{eq:chosen-local} then yield
\[
 \begin{aligned}
 h\sum_{0\le n<N}
 \left\|\norm{w^n}_{H^2}^2\norm{a^n}_2^2\right\|_{L^{4r/3}(\Omega)}
 &\le h\sum_{0\le n<N}\norm{w^n}_{L^{16r/3}(\Omega;H^2)}^2
                   \norm{a^n}_{L^{16r/3}(\Omega;L^2)}^2\\
 &\le CNh\,h^{19/9}\le Ch^2.
 \end{aligned}
\]
Combining these bounds proves \eqref{eq:summed}.
The initial error satisfies
$\norm{e^0}_2^2=\norm{(P_K-I)u_0}_2^2\le CK^{-9}\le Ch^2$.
We also have
\[
 \mathcal E_N\le2\sup_{0\le t\le T}\norm{u_K(t)}_2^2
       +2\max_{0\le n\le N}\norm{U^n}_2^2.
\]
Lemma~\ref{lem:comparison-moments} and Proposition~\ref{prop:numerical} therefore give $\norm{\mathcal E_N}_{L^{2r}(\Omega)}\le C$.

To estimate the martingale term in \eqref{eq:recursion}, set
\[
 \mathfrak M_j=2\sum_{0\le n<j}\ip{e^n}{w^n},\qquad
 \mathfrak M_N^*=\max_{0\le j\le N}|\mathfrak M_j|.
\]
The bounds \eqref{eq:chosen-local} and $\norm{\mathcal E_N}_{L^{2r}(\Omega)}\le C$ imply $\ip{e^n}{w^n}\in L^1(\Omega)$.
By Proposition~\ref{prop:one-step}, $M^n$ is $\F_{t_{n+1}}$-measurable and
$\E(M^n\mid\F_{t_n})=0$.
Since $R^*Q$ is deterministic and $e^n$ is $\F_{t_n}$-measurable,
\[
 \E\bigl(\ip{e^n}{w^n}\mid\F_{t_n}\bigr)
 =\ip{e^n}{R^*Q\E(M^n\mid\F_{t_n})}=0.
\]
Thus $\mathfrak M$ is a martingale with respect to $(\F_{t_j})$.

The sum of squared martingale increments satisfies
\[
 \begin{aligned}
 \sum_{0\le n<N}|\mathfrak M_{n+1}-\mathfrak M_n|^2
 &=4\sum_{0\le n<N}|\ip{e^n}{w^n}|^2\\
 &\le4\sum_{0\le n<N}\norm{e^n}_2^2\norm{w^n}_2^2
 \le4\mathcal E_N\sum_{0\le n<N}\norm{w^n}_2^2.
 \end{aligned}
\]
The discrete BDG inequality and H\"older's inequality give
\begin{align}\label{eq:martingale}
 \norm{\mathfrak M_N^*}_{L^{4r/3}(\Omega)}
 &\le C\left\|\mathcal E_N^{1/2}
              \left(\sum_{0\le n<N}\norm{w^n}_2^2\right)^{1/2}
        \right\|_{L^{4r/3}(\Omega)}\notag\\
 &\le C\norm{\mathcal E_N}_{L^r(\Omega)}^{1/2}
        \left\|\sum_{0\le n<N}\norm{w^n}_2^2\right\|_{L^{2r}(\Omega)}^{1/2}
 \le Ch\norm{\mathcal E_N}_{L^r(\Omega)}^{1/2},
\end{align}
where \eqref{eq:summed} is used in the last inequality.

\medskip\noindent\emph{Step 3: localization and strong error estimate.}
We now sum \eqref{eq:recursion} pathwise. Since
$\norm{e^n}_2^2+\norm{e^{n+1}}_2^2\le2\mathcal E_{n+1}$ and $\gamma_n,B_n\ge0$, taking the maximum of the partial sums yields
\begin{equation}\label{eq:summed-error}
 \mathcal E_j\le\norm{e^0}_2^2+\sum_{0\le n<N}B_n
  +\mathfrak M_N^*+2\sum_{0\le n<j}\gamma_n\mathcal E_{n+1},
 \qquad 0\le j\le N.
\end{equation}
Fix $0<\varepsilon\le\min\{1,T,(32C_*)^{-1}\}$.
For $\varepsilon<h\le1$, Proposition~\ref{prop:numerical} and Theorem~\ref{thm:exact} give
\[
 \E\max_{0\le n\le N}\norm{u(t_n)-U^n}_2^{2r}
 \le C\le C\varepsilon^{-2r}h^{2r}.
\]
Since $\varepsilon$ is fixed, it suffices to consider $0<h\le\varepsilon$. Define
\begin{equation}\label{eq:good-event}
 \Omega_h=\left\{
 \begin{aligned}
 &h\sup_{0\le t\le T}\norm{\partial_x u_K(t)}_\infty\le \varepsilon,\\
 &\sum_{0\le n<N}\norm{\kappa^n}_{H^2}\le \varepsilon,
 \quad h\sum_{0\le n<N}\norm{w^n}_{H^2}\le \varepsilon
 \end{aligned}\right\}=\Omega_{1,h}\cap\Omega_{2,h}\cap\Omega_{3,h},
\end{equation}
where $\Omega_{j,h}$ denotes the event defined by the $j$th condition above. On $\Omega_h$, the definition of $\gamma_n$ gives
\[
 \begin{aligned}
 \max_{0\le n<N}\gamma_n
 &\le C_*(h+3\varepsilon)\le4C_*\varepsilon\le\tfrac18,\\
 \sum_{0\le n<N}\gamma_n
 &\le C_*\left(T+2+\int_0^T\norm{\partial_x u_K(t)}_\infty\dd t\right),
 \end{aligned}
\]
where we used $\varepsilon\le1$ and summed the interval integrals over $[0,t_N]\subset[0,T]$. For $1\le j\le N$, moving $2\gamma_{j-1}\mathcal E_j$ to the left-hand side of \eqref{eq:summed-error} leaves a coefficient $1-2\gamma_{j-1}\ge3/4$. The discrete Gronwall inequality then gives
\[
 \mathcal E_N\le
 \left(\norm{e^0}_2^2+\sum_{0\le n<N}B_n
       +\mathfrak M_N^*\right)
 \prod_{0\le n<N}(1-2\gamma_n)^{-1}.
\]
Since $-\log(1-2\gamma_n)\le2\gamma_n/(1-2\gamma_n)\le(8/3)\gamma_n$, we obtain, on $\Omega_h$,
\begin{equation}\label{eq:path-bound}
 \mathcal E_N\le C_T\exp\left(\frac83C_*\int_0^T\norm{\partial_x u_K(t)}_\infty\dd t\right)
 \left(\norm{e^0}_2^2+\sum_{0\le n<N}B_n+\mathfrak M_N^*\right).
\end{equation}
By \eqref{eq:derivative-exp} with $A=(32r/3)C_*$, the exponential factor in \eqref{eq:path-bound} is bounded in $L^{4r}(\Omega)$, uniformly in $K$ and $\lambda\in[0,\Lambda]$. We apply H\"older's inequality to \eqref{eq:path-bound}, using \eqref{eq:summed} and \eqref{eq:martingale}:
\[
 \begin{aligned}
 \norm{\mathcal E_N\ind_{\Omega_h}}_{L^r(\Omega)}
 &\le C\left(\norm{e^0}_2^2
       +\left\|\sum_{0\le n<N}B_n\right\|_{L^{4r/3}(\Omega)}
       +\norm{\mathfrak M_N^*}_{L^{4r/3}(\Omega)}\right)\\
 &\le Ch^2+Ch\norm{\mathcal E_N}_{L^r(\Omega)}^{1/2}.
 \end{aligned}
\]
We next bound the probability of $\Omega_h^c$. Applying the embedding $H^2\hookrightarrow W^{1,\infty}$ to \eqref{eq:low} and Minkowski's inequality to \eqref{eq:chosen-local} yields
\[
 \begin{aligned}
 \left\|h\sup_{0\le t\le T}\norm{\partial_x u_K(t)}_\infty\right\|_{L^{16r}(\Omega)}&\le Ch,\\
 \left\|\sum_{0\le n<N}\norm{\kappa^n}_{H^2}\right\|_{L^{16r}(\Omega)}&\le CNh^{3/2}\le Ch^{1/2},\\
 \left\|h\sum_{0\le n<N}\norm{w^n}_{H^2}\right\|_{L^{16r}(\Omega)}&\le CNh^{37/18}\le Ch^{19/18}.
 \end{aligned}
\]
Markov's inequality and the union bound give
\begin{equation}\label{eq:bad-prob}
 \PP(\Omega_h^c)
 \le\sum_{j=1}^3\PP(\Omega_{j,h}^c)
 \le C\left(h^{16r}+h^{8r}+h^{152r/9}\right)
 \le Ch^{8r}.
\end{equation}

H\"older's inequality and \eqref{eq:bad-prob} give
\[
 \norm{\mathcal E_N\ind_{\Omega_h^c}}_{L^r(\Omega)}
 \le\norm{\mathcal E_N}_{L^{2r}(\Omega)}\PP(\Omega_h^c)^{1/(2r)}\le Ch^4.
\]
Using $\mathcal E_N=\mathcal E_N\ind_{\Omega_h}+\mathcal E_N\ind_{\Omega_h^c}$, the triangle inequality and Young's inequality give
\[
 \begin{aligned}
 &\norm{\mathcal E_N}_{L^r(\Omega)}
 \le\norm{\mathcal E_N\ind_{\Omega_h}}_{L^r(\Omega)}
      +\norm{\mathcal E_N\ind_{\Omega_h^c}}_{L^r(\Omega)}\\
 \le{}& Ch^2+Ch\norm{\mathcal E_N}_{L^r(\Omega)}^{1/2}+Ch^4
 \le\tfrac12\norm{\mathcal E_N}_{L^r(\Omega)}+Ch^2.
 \end{aligned}
\]
Thus $\norm{\mathcal E_N}_{L^r(\Omega)}\le Ch^2$. Substituting this bound into \eqref{eq:total-error-decomposition} yields
\[
 \E\max_{0\le n\le N}\norm{u(t_n)-U^n}_2^{2r}
 \le C\bigl(h^{2r}+\norm{\mathcal E_N}_{L^r(\Omega)}^r\bigr)
 \le Ch^{2r},\qquad 0<h\le\varepsilon.
\]

\end{proof}

\section{Numerical experiments}\label{sec:numerics}
We examine the strong convergence of the implicit midpoint method for $\lambda\in\{0,0.5,2\}$. Let $\U=\mathbb R^4$ with standard basis $(e_j)_{j=1}^4$, and let $W$ be a standard Wiener process in $\U$. The initial data and noise coefficient are given by
\[
 \begin{gathered}
 u_0(x)=0.4\sin x+0.1\cos(2x),\\
 \Phi e_1=0.5\cos x,\quad \Phi e_2=0.5\sin x,\quad
 \Phi e_3=0.4\cos(2x),\quad \Phi e_4=0.4\sin(2x).
 \end{gathered}
\]
The initial data and all noise profiles have zero mean and belong to $H^s(\T)$ for every $s\ge0$.

We use a Fourier pseudospectral method with 64 equispaced spatial nodes and the $2/3$ rule for dealiasing the quadratic term. The nonlinear equation at each time step is solved by fixed-point iteration with tolerance $10^{-12}$. The computations are performed on $[0,T]$ with $T=1$ and time step sizes $h=2^{-k}$, $k=4,\ldots,8$. The reference solution is computed by the implicit midpoint method on the same spatial grid, with $h_{\rm ref}=2^{-12}$. We estimate the errors using 400 independent sample paths, with each numerical solution and its reference driven by the same Wiener path.

For each $h$ and $\lambda$, let $U^{n,\ell}$ and $U_{\rm ref}^{\ell}(t_n)$ denote the numerical and reference values on the $\ell$th sample path. For $r=1,2$, the $L^{2r}(\Omega)$ norm of the maximum $L^2(\T)$ error at the time points is estimated by
\begin{equation}\label{eq:numerical-error}
 E_r(h,\lambda)=\left(\frac1{400}\sum_{\ell=1}^{400}
 \max_{0\le n\le T/h}\norm{U^{n,\ell}
       -U_{\rm ref}^{\ell}(t_n)}_2^{2r}\right)^{1/(2r)}.
\end{equation}

\begin{figure}[htbp]
 \centering
 \includegraphics[width=\textwidth]{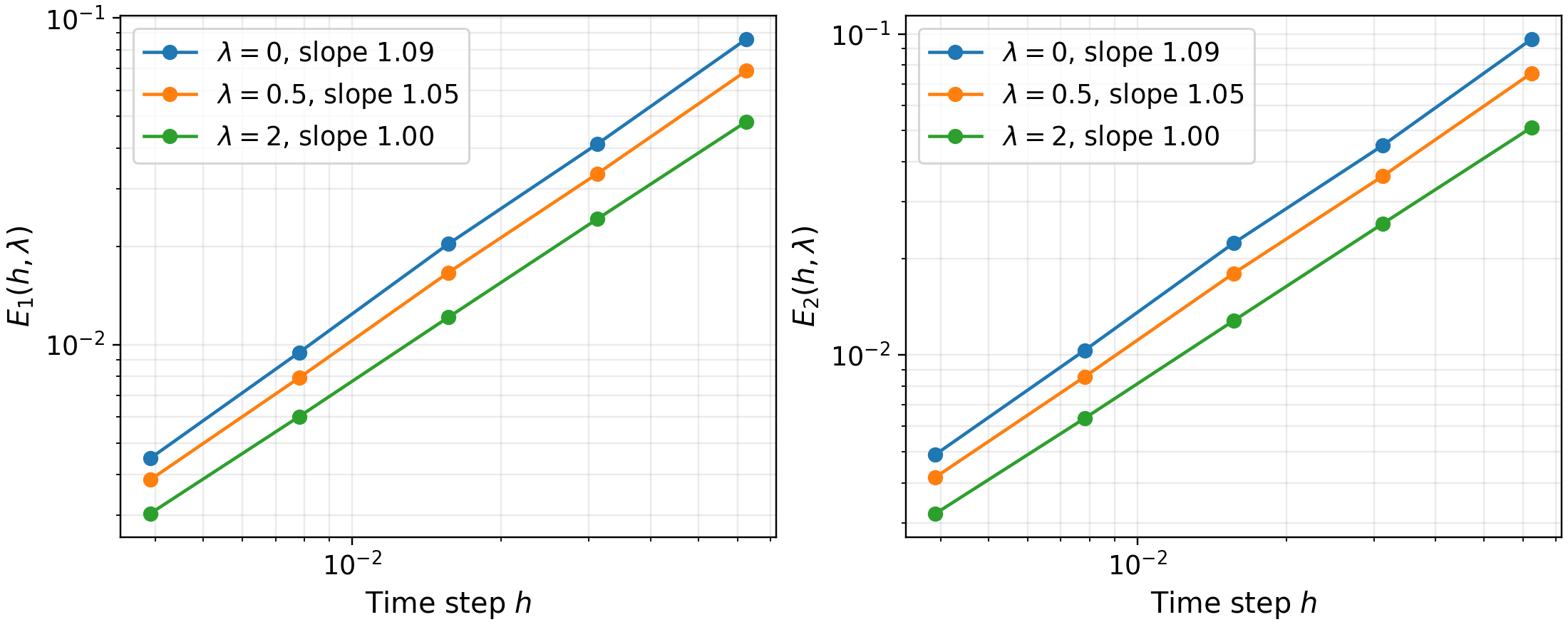}
 \caption{Temporal convergence for $\lambda=0,0.5,2$: $E_1(h,\lambda)$ (left) and $E_2(h,\lambda)$ (right), defined in \eqref{eq:numerical-error}. The legends report convergence rates estimated using $h=2^{-6},2^{-7},2^{-8}$.}
 \label{fig:midpoint-convergence}
\end{figure}

Figure~\ref{fig:midpoint-convergence} shows rates close to one for both error measures and all three damping coefficients, including $\lambda=0$.

\FloatBarrier
\section{Conclusion}\label{sec:conclusion}
We have established first-order strong convergence of the implicit midpoint method for the periodic stochastic KdV equation with additive noise. On any fixed finite time interval, the maximum $L^2(\T)$ error at the time points is $O(h)$ in $L^p(\Omega)$ for all $1\le p<\infty$. The estimate holds for every solution of the implicit midpoint scheme and any damping coefficient $\lambda\ge0$, including the undamped case $\lambda=0$. This answers, in the periodic setting, the question of a strong convergence rate that remained open after the work of Debussche and Printems. Using a classical KdV invariant and the coercivity estimate, we derive uniform exponential bounds for comparison solutions and pass to the limit to obtain exponential integrability of the exact solution. The strong convergence proof combines estimates for the comparison solutions with an $L^2$ error propagation identity that isolates a martingale increment. The numerical results support the predicted first-order convergence. Future work will focus on fully discrete approximations.

\clearpage
\appendix
\section{Proof of Proposition~\ref{prop:smooth}}\label{app:comparison-existence}

\begin{proof}
\emph{Step 1: local existence, uniqueness, and adaptedness.}
For fixed $K$, $\eta_K(t)=\int_0^tP_K\Phi\dd W$ has continuous paths in $P_KH$, hence in every $H^m$ on finite time intervals, almost surely. Setting $v_K=u_K-\eta_K$ and subtracting $\mathrm d\eta_K=P_K\Phi\dd W$ from \eqref{eq:smooth} gives
\begin{equation}
\partial_t v_K+\partial_x^3v_K+v_K\partial_x v_K
+\partial_x(\eta_Kv_K)+\lambda v_K
=
-\partial_x^3\eta_K-\eta_K\partial_x \eta_K-\lambda\eta_K,
\label{eq:pathwise}
\end{equation}
with \(v_K(0)=P_Ku_0\).

For fixed \(K\) and an integer \(\ell\ge K\), let \(v_{K,\ell}\) take
values in \(P_\ell H\) and solve
\begin{equation}\label{eq:comparison-galerkin}
\partial_t v_{K,\ell}
+\partial_x^3v_{K,\ell}
+P_\ell(v_{K,\ell}\partial_x v_{K,\ell})
+P_\ell\partial_x(\eta_Kv_{K,\ell})
+\lambda v_{K,\ell}
=
P_\ell f_K,
\end{equation}
where \(f_K=-\partial_x^3\eta_K-\eta_K\partial_x \eta_K-\lambda\eta_K\)
and \(v_{K,\ell}(0)=P_Ku_0\).
By standard finite-dimensional ODE theory, this system has a unique local solution.

Let \(m\ge3\). Since \(P_\ell\) is an orthogonal projection in \(H^m\),
taking the \(H^m\) inner product of \eqref{eq:comparison-galerkin} with
\(v_{K,\ell}\) and using the skew-adjointness of \(\partial_x^3\) gives
\[
\begin{aligned}
\frac12\frac{\mathrm d}{\mathrm dt}\|v_{K,\ell}\|_{H^m}^2
+\lambda\|v_{K,\ell}\|_{H^m}^2
={}&
-\langle \langle D\rangle^m(v_{K,\ell}\partial_x v_{K,\ell}),
         \langle D\rangle^m v_{K,\ell}\rangle
\\
&-\langle \langle D\rangle^m\partial_x(\eta_Kv_{K,\ell}),
          \langle D\rangle^m v_{K,\ell}\rangle
\\
&+\langle \langle D\rangle^m f_K,
          \langle D\rangle^m v_{K,\ell}\rangle .
\end{aligned}
\]
The first term is bounded by \eqref{eq:tame}. For the mixed term, set \(v=v_{K,\ell}\) and write
\[
\langle D\rangle^m\partial_x(\eta_Kv)
=
\eta_K\partial_x(\langle D\rangle^m v)
+
[\langle D\rangle^m,\eta_K]\partial_x v
+
\langle D\rangle^m((\partial_x\eta_K)v).
\]
Periodic integration by parts gives
\[
\langle
\eta_K\partial_x(\langle D\rangle^m v),
\langle D\rangle^m v
\rangle
=
-\frac12\int_{\mathbb T}
(\partial_x\eta_K)|\langle D\rangle^m v|^2.
\]
The commutator estimate \eqref{eq:periodic-commutator}, the product estimate
\eqref{eq:product}, and the embedding \(H^m\hookrightarrow W^{1,\infty}\)
therefore imply
\[
\|[\langle D\rangle^m,\eta_K]\partial_x v\|_2
+
\|(\partial_x\eta_K)v\|_{H^m}
\le
C_m\|\eta_K\|_{H^{m+1}}\|v\|_{H^m},
\]
and hence
\[
\big|
\langle
\langle D\rangle^m\partial_x(\eta_Kv),
\langle D\rangle^m v
\rangle
\big|
\le
C_m\|\eta_K\|_{H^{m+1}}\|v\|_{H^m}^2.
\]
The product estimate also gives
\[
\begin{aligned}
\|f_K\|_{H^m}^2
&\le C_{m,\lambda}\left(\|\eta_K\|_{H^{m+3}}^2
      +\|\eta_K\|_{H^{m+1}}^4\right)\\
&\le C_{m,\lambda}\left(1+\|\eta_K\|_{H^{m+3}}^4\right).
\end{aligned}
\]
Set \(E_{m,\ell}=1+\|v_{K,\ell}\|_{H^m}^2\).
Combining the bounds for the three terms with
\(\|\partial_x v_{K,\ell}\|_\infty\le C\|v_{K,\ell}\|_{H^3}\)
and Young's inequality gives
\begin{equation}
\frac{\mathrm d}{\mathrm dt}E_{m,\ell}
\le
C_{m,\lambda}
\left(
1+\|v_{K,\ell}\|_{H^3}
+\|\eta_K\|_{H^{m+3}}^4
\right)
E_{m,\ell},
\qquad m\ge3,
\label{eq:local-high-persistence}
\end{equation}
where \(C_{m,\lambda}\) is independent of \(\ell\).

Fix \(T>0\). Since
\(\sup_{0\le t\le T}\|\eta_K(t)\|_{H^6}<\infty\)
and \(\|v_{K,\ell}\|_{H^3}\le E_{3,\ell}^{1/2}\),
estimate \eqref{eq:local-high-persistence} with $m=3$ gives $\frac{\mathrm d}{\mathrm dt}E_{3,\ell}\le CE_{3,\ell}^{3/2}$ for $t\le T$ as long as $v_{K,\ell}$ exists.
Here \(C\) depends only on \(\lambda\) and
\(\sup_{0\le t\le T}\|\eta_K(t)\|_{H^6}\).
The initial value \(E_{3,\ell}(0)=1+\|P_Ku_0\|_{H^3}^2\)
is independent of \(\ell\). Choose \(0<t_*\le T\) such that
\(Ct_*\sqrt{1+\|P_Ku_0\|_{H^3}^2}\le1\).
Integrating on the interval of existence gives, for \(t\le t_*\),
\[
E_{3,\ell}(t)
\le
\frac{E_{3,\ell}(0)}
{\left(1-\frac C2t\sqrt{E_{3,\ell}(0)}\right)^2}
\le4E_{3,\ell}(0).
\]
The continuation criterion for finite-dimensional ODEs therefore ensures
that every \(v_{K,\ell}\) exists on \([0,t_*]\), with an
\(H^3\) bound independent of \(\ell\).

Applying Gronwall's inequality to \eqref{eq:local-high-persistence},
using the uniform \(H^3\) bound and the regularity of \(\eta_K\), gives
\begin{equation}\label{eq:comparison-uniform-sobolev}
\sup_{\ell\ge K}\sup_{0\le t\le t_*}
\|v_{K,\ell}(t)\|_{H^m}<\infty,
\qquad m\ge3.
\end{equation}
Using \eqref{eq:comparison-uniform-sobolev} with \(m\) replaced by \(m+3\), together with
\eqref{eq:comparison-galerkin} and the product estimate \eqref{eq:product}, gives
\[
\sup_{\ell\ge K}\sup_{0\le t\le t_*}
\left(
\|v_{K,\ell}(t)\|_{H^{m+1}}
+
\|\partial_t v_{K,\ell}(t)\|_{H^m}
\right)
<\infty,
\qquad m\ge3.
\]
These bounds give equicontinuity in $H^m$ and, by $H^{m+1}\compactembed H^m$, relative compactness of the values. Arzelà--Ascoli and a diagonal argument yield a subsequence $\ell_j\to\infty$ and a function $v_K$ such that $v_{K,\ell_j}\to v_K$ in $C([0,t_*];H^m)$ for every integer $m\ge3$.

Convergence in \(C([0,t_*];H^{m+3})\), the product estimate
\eqref{eq:product}, and the Fourier projection estimate
\(\|(I-P_\ell)g\|_{H^m}\le C\ell^{-1}\|g\|_{H^{m+1}}\)
for \(g\in H_0^{m+1}\) allow us to pass to the limit in the integral form of
\eqref{eq:comparison-galerkin}. Hence \(v_K\) satisfies
\eqref{eq:pathwise} with \(v_K(0)=P_Ku_0\).
Consequently, \(u_K=v_K+\eta_K\) satisfies the integral form of
\eqref{eq:smooth} on \([0,t_*]\).

We next prove uniqueness. Let \(v,\widetilde v\in C([0,t_*];H^3)\)
be two solutions of \eqref{eq:pathwise} with the same initial value and set
\(z=v-\widetilde v\). Subtracting the equations gives
\[
\partial_tz+\partial_x^3z
+\partial_x\big((\widetilde v+\eta_K)z\big)
+z\partial_x z+\lambda z=0.
\]
Taking the \(L^2\) inner product with \(z\) and integrating by parts, we obtain
\[
\frac{\mathrm d}{\mathrm dt}\|z\|_2^2
+
2\lambda\|z\|_2^2
=
-\int_{\mathbb T}
\partial_x(\widetilde v+\eta_K)z^2
\le
\|\partial_x(\widetilde v+\eta_K)\|_\infty
\|z\|_2^2.
\]
Since \(H^3\hookrightarrow W^{1,\infty}\), Gronwall's inequality gives
\(z=0\). Uniqueness and relative compactness then imply
\(v_{K,\ell}\to v_K\) in \(C([0,t_*];H^m)\) for every integer \(m\ge3\).

The local existence argument also applies to \eqref{eq:pathwise} with
initial value \(v_K(s_0)=\phi\) at any time \(s_0\ge0\), provided that
\(\phi\in H_0^m\) for every integer \(m\ge3\).
Solve \eqref{eq:comparison-galerkin} with initial value
\(v_{K,\ell}(s_0)=P_\ell\phi\).
Since \(\|P_\ell\phi\|_{H^m}\le\|\phi\|_{H^m}\) and
\(P_\ell\phi\to\phi\) in \(H^m\), the estimates \eqref{eq:local-high-persistence}--\eqref{eq:comparison-uniform-sobolev} and the
compactness argument remain valid.
The length of the local existence interval depends only on \(\lambda\),
\(\|\phi\|_{H^3}\) and
\(\sup_{s_0\le s\le s_0+1}\|\eta_K(s)\|_{H^6}\).
By uniqueness, these local solutions define a unique maximal solution on $[0,\zeta_K)$, where $\zeta_K$ is the supremum of the times of extension, with $v_K\in C([0,t];H_0^m)$ for every $t<\zeta_K$ and integer $m\ge3$.

For fixed $t\ge0$, we claim that $t<\zeta_K$ if and only if a subsequence of solutions of \eqref{eq:comparison-galerkin} exists on $[0,t]$ and is bounded in $C([0,t];H^3)$.

Suppose that \(0<t<\zeta_K\) and set
\(\varrho=1+\sup_{0\le r\le t}\|v_K(r)\|_{H^3}\).
Suppose that \(v_{K,\ell}\) exists on
\([0,s]\), where \(0<s\le t\), and satisfies
\(\sup_{0\le r\le s}\|v_{K,\ell}(r)\|_{H^3}\le2\varrho\).
Gronwall's inequality applied to \eqref{eq:local-high-persistence}
then bounds \(v_{K,\ell}\) in every \(H^m\) on \([0,s]\), with constants
independent of \(\ell\) and \(s\).
Set \(z=v_{K,\ell}-v_K\). Since
\(v_{K,\ell}=P_\ell v_{K,\ell}\), we have
\(P_\ell z=z+(I-P_\ell)v_K\).
Subtracting \eqref{eq:pathwise} from \eqref{eq:comparison-galerkin}
and taking the \(L^2\) inner product with \(z\) gives
\[
\begin{aligned}
\frac{\mathrm d}{\mathrm dr}\|z\|_2^2+2\lambda\|z\|_2^2
={}&-\int_{\mathbb T}\partial_x(v_K+\eta_K)z^2\\
&-2\left\langle
v_{K,\ell}\partial_x v_{K,\ell}
+\partial_x(\eta_Kv_{K,\ell})-f_K,\,
(I-P_\ell)v_K
\right\rangle .
\end{aligned}
\]
The last term has absolute value at most
\(C\ell^{-6}\sup_{0\le r\le t}\|v_K(r)\|_{H^6}\),
where \(C\) is independent of \(\ell\) and \(s\).
Since \(z(0)=0\), Gronwall's inequality yields
\(\sup_{0\le r\le s}\|z(r)\|_2\le C\ell^{-3}\).
Combining this bound with
\(\|z\|_{H^3}^2\le\|z\|_2\|z\|_{H^6}\)
and the uniform \(H^6\) bounds gives
\begin{equation}\label{eq:comparison-galerkin-error}
\sup_{0\le r\le s}\|v_{K,\ell}(r)-v_K(r)\|_{H^3}
\le C\ell^{-3/2},
\end{equation}
with \(C\) independent of \(\ell\) and \(s\).
For all sufficiently large \(\ell\), the right-hand side is at most
\(1/2\), so
\(\sup_{0\le r\le s}\|v_{K,\ell}(r)\|_{H^3}\le\varrho-1/2<2\varrho\).
A bootstrap argument using continuity and the continuation criterion for
\eqref{eq:comparison-galerkin} shows that \(v_{K,\ell}\) exists on \([0,t]\).
Taking \(s=t\) in \eqref{eq:comparison-galerkin-error} gives
\begin{equation}\label{eq:comparison-galerkin-convergence}
v_{K,\ell}\longrightarrow v_K
\qquad\text{in }C([0,t];H^3).
\end{equation}

Conversely, suppose that a subsequence of solutions of
\eqref{eq:comparison-galerkin} exists on \([0,t]\) and is bounded in
\(C([0,t];H^3)\).
Estimate \eqref{eq:local-high-persistence} gives uniform bounds in every
\(H^m\), so the compactness argument for \eqref{eq:comparison-galerkin} yields a solution of
\eqref{eq:pathwise} in \(C([0,t];H_0^m)\) for every integer \(m\ge3\).
By uniqueness, this solution agrees with \(v_K\) wherever both are defined.
Local continuation at time \(t\) then gives \(t<\zeta_K\).

For fixed \(K,\ell\) and \(t\ge0\), the solution of \eqref{eq:comparison-galerkin} on \([0,t]\) depends continuously on the input path \(\eta_K|_{[0,t]}\in C([0,t];P_KH)\), and the set of input paths for which it exists on \([0,t]\) is open.
Since \(\eta_K|_{[0,t]}\) is \(\mathcal F_t\)-measurable, the events
\[
E_{\ell,n}(t)
=
\left\{
v_{K,\ell}\text{ exists on }[0,t],\quad
\sup_{0\le s\le t}\|v_{K,\ell}(s)\|_{H^3}\le n
\right\}
\]
belong to \(\mathcal F_t\) for every \(n\in\mathbb N\).
The characterization of \(\{t<\zeta_K\}\) gives
\[
\{\zeta_K>t\}
=
\bigcup_{n=1}^{\infty}
\bigcap_{N=K}^{\infty}
\bigcup_{\ell\ge N}
E_{\ell,n}(t)
\in\mathcal F_t,
\]
so \(\zeta_K\) is a stopping time.

After extending \(v_{K,\ell}(t)\) by zero where it is undefined,
\eqref{eq:comparison-galerkin-convergence} implies that
\(v_K(t)\mathbf1_{\{t<\zeta_K\}}\) is \(\mathcal F_t\)-measurable.
Hence \(u_K=v_K+\eta_K\) is adapted on \([0,\zeta_K)\).

\medskip\noindent\emph{Step 2: global existence.}
Suppose that \(\zeta_K<\infty\) and
\(\sup_{t<\zeta_K}\|v_K(t)\|_{H^3}<\infty\).
Since \(\eta_K\) is bounded in \(H^6\) on \([0,\zeta_K+1]\),
Step 1 provides a local existence time bounded below independently
of the starting time \(s_0<\zeta_K\).
Taking \(s_0\) sufficiently close to \(\zeta_K\) contradicts maximality.
Since \(u_K=v_K+\eta_K\) and \(\eta_K\) is bounded in \(H^3\)
on finite intervals, we obtain
\begin{equation}\label{eq:comparison-blowup-alternative}
\zeta_K<\infty
\quad\Longrightarrow\quad
\limsup_{t\uparrow\zeta_K}\|u_K(t)\|_{H^3}=\infty.
\end{equation}

We first show that finite-time blow-up forces \(G(u_K)\) to become unbounded.
Fix \(T>0\), \(R>G(P_Ku_0)\) and \(M>\|P_Ku_0\|_{H^6}\)
and define the stopping times
\[
\tau_R
=
\inf\{t<\zeta_K:G(u_K(t))\ge R\},
\qquad
\sigma_M
=
\inf\{t<\zeta_K:\|u_K(t)\|_{H^6}\ge M\},
\]
where \(\inf\varnothing=\infty\).
Since \(\|u_K\|_{H^3}\le\|u_K\|_{H^6}\),
\eqref{eq:comparison-blowup-alternative} implies that
\(\sigma_M<\zeta_K\) on \(\{\zeta_K<\infty\}\)
for every \(M>\|P_Ku_0\|_{H^6}\).

For \(L>\|P_Ku_0\|_{H^9}\), set
\(\rho_L=\inf\{t<\zeta_K:\|u_K(t)\|_{H^9}\ge L\}\).
On \([0,T\wedge\tau_R\wedge\sigma_M\wedge\rho_L]\),
\(u_K\) is bounded in \(H^9\), so Itô's formula applies to
\(G(u_K)\) and \(\|u_K\|_{H^6}^2\).

Since \(T\wedge\sigma_M<\zeta_K\), continuity in \(H^9\) gives
\(\sup_{0\le t\le T\wedge\sigma_M}\|u_K(t)\|_{H^9}<\infty\) almost surely.
Thus \(\rho_L>T\wedge\sigma_M\) for all sufficiently large \(L\), almost surely.

Applying Lemma~\ref{lem:generator} with \(\Phi\) replaced by \(P_K\Phi\) gives
\[
\mathrm dG(u_K)
\le
\left(
C_0\|P_K\Phi\|_{\mathcal L_2(\U,H^2)}^2
+
C_1\lambda G(u_K)
\right)\dd t
+
\mathrm DG(u_K)[P_K\Phi\dd W].
\]
Before \(\tau_R\), the quadratic variation of the stopped stochastic
integral is bounded by
\(C_2\|P_K\Phi\|_{\mathcal L_2(\U,H^2)}^2R^{5/6}T\).
The stopped stochastic integral is therefore a square-integrable
martingale. The stopped values of \(G\) are bounded by \(R\).
Taking expectations, applying Gronwall's inequality, and letting
\(L\to\infty\) by bounded convergence, we obtain, for \(0\le t\le T\),
\begin{equation}
\mathbb E
G\bigl(u_K(t\wedge\tau_R\wedge\sigma_M)\bigr)
\le
\left(
G(P_Ku_0)
+
C_0\|P_K\Phi\|_{\mathcal L_2(\U,H^2)}^2t
\right)
e^{C_1\lambda t}.
\label{eq:stoppedG}
\end{equation}

We next remove the stopping time \(\sigma_M\). Set
\(E_6(t)=1+\|u_K(t)\|_{H^6}^2\),
\(\beta_{6,K}^2=\|P_K\Phi\|_{\mathcal L_2(\U,H^6)}^2\),
and, for \(t<\zeta_K\),
\[
M_6(t)
=
2\int_0^t
\left\langle
\langle D\rangle^6u_K(s),
\langle D\rangle^6P_K\Phi\dd W(s)
\right\rangle .
\]
The skew-adjointness of \(\partial_x^3\), the estimate \eqref{eq:tame},
and the nonpositivity of the damping contribution give
\[
\mathrm dE_6
\le
C_6\|\partial_x u_K\|_\infty E_6\dd t
+
\beta_{6,K}^2\dd t
+
\mathrm dM_6.
\]
Moreover,
\[
\mathrm d[M_6]_t
=
4\|
(\langle D\rangle^6P_K\Phi)^*
\langle D\rangle^6u_K(t)
\|_{\U}^2\dd t
\le
4\beta_{6,K}^2E_6(t)\dd t.
\]

Define \(N_6(t)=\int_0^t E_6(s)^{-1}\dd M_6(s)\).
Since \(E_6\ge1\), Itô's formula gives
\begin{equation}
\mathrm d\log E_6
\le
\left(
C_6\|\partial_x u_K\|_\infty+\beta_{6,K}^2
\right)\dd t
+
\mathrm dN_6,
\label{eq:continuation-log}
\end{equation}
and
\[
\mathrm d[N_6]_t
=
E_6(t)^{-2}\mathrm d[M_6]_t
\le
4\beta_{6,K}^2\dd t.
\]
Extending the integrand by zero after \(\zeta_K\wedge\tau_R\)
gives a square-integrable martingale on \([0,T]\), which we again denote by \(N_6\).
Doob's inequality yields
\begin{equation}\label{eq:doob-N6}
\mathbb E\sup_{0\le t\le T}|N_6(t)|^2
\le
4\mathbb E[N_6]_T
\le
16\beta_{6,K}^2T.
\end{equation}

For \(t<\zeta_K\wedge\tau_R\), Lemma~\ref{lem:derivative-control} gives
\(\|\partial_x u_K(t)\|_\infty\le CR\).
Integrating \eqref{eq:continuation-log} up to
\(t\wedge\tau_R\wedge\sigma_M\) therefore yields
\[
\log E_6(t\wedge\tau_R\wedge\sigma_M)
\le
\log E_6(0)
+
(CR+\beta_{6,K}^2)T
+
\sup_{0\le s\le T}|N_6(s)|.
\]

On the event \(\{\sigma_M\le T\wedge\tau_R\}\),
continuity of \(u_K\) in \(H^6\) gives \(E_6(\sigma_M)=1+M^2\).
Hence
\[
\{\sigma_M\le T\wedge\tau_R\}
\subset
\left\{
\sup_{0\le s\le T}|N_6(s)|
\ge
\log\frac{1+M^2}{E_6(0)}
-
(CR+\beta_{6,K}^2)T
\right\}.
\]
For all sufficiently large \(M\), Chebyshev's inequality and
\eqref{eq:doob-N6} give
\begin{equation}
\mathbb P(\sigma_M\le T\wedge\tau_R)
\le
\frac{
16\beta_{6,K}^2T
}{
\left[
\log\frac{1+M^2}{E_6(0)}
-
(CR+\beta_{6,K}^2)T
\right]^2
}
\longrightarrow0
\qquad
\text{as }M\to\infty.
\label{eq:sobolev-exit-probability}
\end{equation}

On \(\{\zeta_K\le T,\ \tau_R\ge\zeta_K\}\),
the blow-up alternative \eqref{eq:comparison-blowup-alternative} and
\(\|u_K\|_{H^3}\le\|u_K\|_{H^6}\) imply
\(\sigma_M<\zeta_K\le T\wedge\tau_R\)
for every \(M>\|P_Ku_0\|_{H^6}\). Thus
\[
\{\zeta_K\le T,\ \tau_R\ge\zeta_K\}
\subset
\{\sigma_M\le T\wedge\tau_R\}.
\]
By \eqref{eq:sobolev-exit-probability}, $\mathbb P(\zeta_K\le T,\tau_R\ge\zeta_K)=0$, and hence $T\wedge\tau_R<\zeta_K$ almost surely.
Continuity of \(u_K\) in \(H^6\) then gives
\(\sigma_M>T\wedge\tau_R\) for all sufficiently large \(M\), almost surely.
Fatou's lemma and \eqref{eq:stoppedG} with \(t=T\) yield
\begin{equation}\label{eq:stopped-G-limit}
\begin{aligned}
\mathbb E G(u_K(T\wedge\tau_R))
&\le
\liminf_{M\to\infty}
\mathbb E
G(u_K(T\wedge\tau_R\wedge\sigma_M))
\\
&\le
\left(
G(P_Ku_0)
+
C_0\|P_K\Phi\|_{\mathcal L_2(\U,H^2)}^2T
\right)
e^{C_1\lambda T}.
\end{aligned}
\end{equation}

By continuity, \(G(u_K(\tau_R))=R\) on \(\{\tau_R\le T\}\).
Since \(\mathbb P(\zeta_K\le T,\tau_R\ge\zeta_K)=0\),
Markov's inequality and \eqref{eq:stopped-G-limit} give
\[
\mathbb P(\zeta_K\le T)
\le \mathbb P(\tau_R\le T)
\le \frac{\mathbb E G(u_K(T\wedge\tau_R))}{R}
\longrightarrow0
\qquad\text{as }R\to\infty.
\]
Since \(T>0\) is arbitrary, \(\zeta_K=\infty\) almost surely.
The regularity follows from Step~1.
\end{proof}

\end{document}